\documentclass[reqno]{amsart}

\usepackage{amsmath,times}
\usepackage{amsthm,amssymb}
\usepackage[all]{xy}
\usepackage{latexsym}
\usepackage{stackrel}
\usepackage{graphicx}
\usepackage{color}
\usepackage{enumerate}
\usepackage{float}
\usepackage{array}
\newcolumntype{L}{>{$}l<{$}} 

\usepackage{comment}
\usepackage[textsize=tiny]{todonotes}

\usepackage{hyperref}
\hypersetup{
    colorlinks=true, 
    linktoc=all,     
    linkcolor=blue,  
}
\theoremstyle{definition}

\newtheorem{theorem}{Theorem}[section]
\newtheorem{lemma}[theorem]{Lemma}
\newtheorem{claim}[theorem]{Claim}
\newtheorem{definitionlemma}[theorem]{Definition-Lemma}

\newtheorem{proposition}[theorem]{Proposition}

\theoremstyle{definition}
\newtheorem{example}[theorem]{Example}
\newtheorem{definition}[theorem]{Definition}
\newtheorem{corollary}[theorem]{Corollary}
\newtheorem{remark}[theorem]{Remark}
\newtheorem{question}[theorem]{Question}
\newtheorem{remarks}[theorem]{Remarks}

\numberwithin{equation}{section}

\newcommand{\C}{\mathbb{C}}
\newcommand{\Nef}{\operatorname{Nef}}

\newcommand{\N}{\mathbb{N}}
\renewcommand{\P}{\mathbb{P}}
\newcommand{\Q}{\mathbb{Q}}
\newcommand{\R}{\mathbb{R}}
\newcommand{\Z}{\mathbb{Z}}
\newcommand{\sub}{\operatorname{sub}}
\newcommand{\Hom}{\operatorname{Hom}}

\newcommand{\cO}{\mathcal{O}}
\newcommand{\rank}{\operatorname{rk}}

\title[{H}odge-{R}iemann bilinear relations for {S}chur classes of ample vector bundles]{{H}odge-{R}iemann bilinear relations for \\ {S}chur classes of ample vector bundles}

\author{Julius Ross and Matei Toma}

\address{Department of Mathematics, Statistics, and Computer Science, University of Illinois at Chicago, 322 Science and Engineering Offices (M/C 249), 851 S. Morgan Street, Chicago, IL 60607
 }
\email{juliusro@uic.edu}

\address{Universit\'e de Lorraine, CNRS, IECL, F-54000 Nancy, France
 }
\email{matei.toma@univ-lorraine.fr}

\date{\today}

\begin{document}
\begin{abstract}Let $X$ be a $d$ dimensional projective manifold, $E$ be an ample vector bundle on $X$ and $0\le \lambda_N\le \lambda_{N-1} \le \cdots \le \lambda_1 \le \rank(E)$  be a partition of $d-2$.  We prove that the Schur class $s_{\lambda}(E)\in H^{d-2,d-2}(X)$ has the Hard Lefschetz property and 
satisfies the Hodge-Riemann bilinear relations.  As a consequence we obtain various new inequalities between characteristic classes of ample vector bundles, including a higher-rank version of the Khovanskii-Teissier inequalities.
\end{abstract}

\maketitle

\section{Introduction}

As is well known, Hodge Theory on projective manifolds has a number of deep topological consequences.   
The two basic examples of this are the Hard Lefschetz Theorem which implies that if $L$ is an ample line bundle on a projective manifold $X$ of dimension $d$, and $k\le d$ is chosen so $d-k$ is even then the map
$$ H^{\frac{d-k}{2},\frac{d-k}{2}}(X;\mathbb R) \xrightarrow{\wedge c_1(L)^k} H^{\frac{d+k}{2},\frac{d+k}{2}}(X;\mathbb R)$$
is an isomorphism, and the Hodge-Riemann bilinear relations which state that the bilinear form
$$ (\alpha,\alpha') \mapsto (-1)^{\frac{d-k}{2}}\int_X \alpha c_1(L)^k \alpha' \text{ for } \alpha,\alpha'\in H^{\frac{d-k}{2},\frac{d-k}{2}}(X;\mathbb R)$$
is positive definite on the primitive cohomology $$H^{\frac{d-k}{2},\frac{d-k}{2}}_p(X;\mathbb R):=\{\alpha:  \alpha\wedge c_1(L)^{k+1}=0\}.$$

Given the importance of these results it is natural to question if these properties continue to hold when $c_1(L)^k$ is replaced by some other class in $H^{k,k}(X;\mathbb R)$.  One result in this direction is that of Bloch-Gieseker \cite{BlochGieseker} which implies that if $E$ is an ample vector bundle of rank $e\le d$ on $X$ with $d-e$ even then $c_e(E)$ has the Hard Lefschetz property, i.e.\ the map 
$$ H^{\frac{d-e}{2},\frac{d-e}{2}}(X;\mathbb R) 
\xrightarrow{\wedge c_e(E)}
H^{\frac{d+e}{2},\frac{d+e}{2}}(X;\mathbb R)$$
is an isomorphism. 

The main result of this paper extends this statement, when $e=d-2$,  to show that in fact the Hodge-Riemann bilinear relations also hold for $c_e(E)$, and furthermore generalizes it to all Schur classes. This is the following


\begin{theorem}[= Theorem \ref{thm:schurhodgeriemann}]\label{thm:intromain}
Let $E$ be a rank $e$ ample vector bundle on a projective manifold $X$ of dimension $d$, let $h\in H^{1,1}(X,\mathbb Z)$ be an ample class and set $c_i:=c_i(E)$.     Given $$0\le \lambda_N\le \lambda_{N-1}\le \cdots \le\lambda_1\le e$$ with $\sum_i \lambda_i =d-2$ consider the Schur class
$$s_{\lambda}(E) = \det \left(\begin{array}{cccc}c_{\lambda_1} & c_{\lambda_1+1} & \cdots & c_{\lambda_1+N-1}\\
c_{\lambda_2-1} & c_{\lambda_2} & \cdots & c_{\lambda_2+N-2}\\
\ldots & \ldots & \ldots & \ldots  \\
c_{\lambda_N-N+1} & c_{\lambda_N-N+2} & \cdots & c_{\lambda_N}\\
\end{array}\right) \in H^{d-2,d-2}(X,\mathbb R).$$
Then
\begin{enumerate}
\item The Hard Lefschetz Property holds for $s_{\lambda}(E)$.  That is, the map 
$$ H^{1,1}(X;\mathbb R) \to H^{d-1,d-1}(X;\mathbb R) \quad \alpha \mapsto \alpha \wedge s_{\lambda}(E)$$
is an isomorphism.
\item The Hodge-Riemann bilinear relations hold for $s_{\lambda}(E)$.  That is, the intersection pairing
$$ (\alpha,\alpha') \mapsto \int_X \alpha s_{\lambda}(E) \alpha' \text{ for } \alpha,\alpha'\in H^{1,1}(X;{\mathbb R})$$
is negative definite on the primitive cohomology
$$H^{1,1}_{p,\lambda}(X;{\mathbb R})  : =\{ \alpha : \int_X \alpha\wedge s_{\lambda}(E)\wedge h=0\}.$$
\end{enumerate}
\end{theorem}

The above theorem is in the same spirit as that of Fulton-Lazarsfeld \cite{FultonLazarsfeld} who consider such Schur classes when $\sum_i \lambda_i=d$ and prove that if $E$ is ample then $\int_X s_{\lambda}(E)>0$.    From this point of view one can also view Theorem \ref{thm:intromain} as a statement about 
positivity properties enjoyed by ample vector bundles.

As an application we partially answer a question posed by Debarre-Ein-Lazarsfeld-Voisin \cite{DELV11} (also Lehmann-Fulger \cite{Lehmann}) concerning the relation between the cone spanned by Schur classes of nef bundles and the cone of positive higher codimensional cycles.   In summary, we show that the former cone is strictly contained in the nef cone of codimension 2 cycles on the product of a very general  principally polarized  abelian surface with itself.

\begin{center}
*
\end{center}

The classical Hodge-Riemann bilinear relations  are known to imply the Hodge-Index inequality as well as many generalisations, and wrapped up in our account of Theorem \ref{thm:intromain} are a number of similar such inequalities.  We list two now, the second of which is particularly striking.

\begin{theorem}[= Theorem \ref{thm:HI2}]\label{thm:HI2intro}
Let $X$ be a projective manifold of dimension $d\ge2$, let $E$ be an ample bundle on $X$  with $\rank(E)\ge d-1$ and let $h$ be an ample class on $X$. Then for any $\alpha\in H^{1,1}(X;\mathbb R)$ 
\begin{equation}\int_{X} \alpha^2 c_{d-2}(E) \int_{X} h c_{d-1}(E) \le 2
\int_{X} \alpha c_{d-2}(E) h \int_X \alpha c_{d-1}(E)\label{eq:introineq2}\end{equation}
with equality if and only if $\alpha=0$.
\end{theorem}

Notice also that \eqref{eq:introineq2} implies that the bilinear form $(\alpha,\alpha')\mapsto \int_X \alpha c_{d-2}(E)\alpha'$  is negative definite on the subspace $\{ \alpha : \int_X \alpha c_{d-1}(E)=0\}$ (from which the Hodge-Riemann bilinear relations follow easily). 

\begin{theorem}[= Theorem \ref{thm:logconcaveschur}]\label{thm:intrologconvex}
Let $X$ be a projective manifold of dimension $d$, let $E$ be an ample bundle on $X$  with $\rank(E)\ge d$ and let $h$ be an ample class on $X$.  Then the map
$$ i\mapsto \int_X c_i(E) h^{d-i} \text{ for } i=0,\ldots, d$$
is strictly log-concave.
That is, given integers $0\le i<j<k\le d$ and defining $t$ so
$$ ti + (1-t)k=j$$
we have
$$ t \log \int_X c_i(E) h^{d-i} + (1-t) \log \int_X c_{k}(E) h^{d-k} < \log \int_X c_{j}(E) h^{d-i}.$$
\end{theorem}

 One should think of this statement a higher rank version of the famous Khovanskii-Teissier inequalities (see Remark \ref{rem:Khovanskii}).\\

It is possible to generalise this log-concavity to other Schur classes as follows.   For any partition $\mu$ the Schur polynomial  $s_{\mu}(x_1,\ldots,x_e)$ is a symmetric polynomial,  from which we may define new symmetric polynomials $s_{\mu}^{(i)}$ by requiring
$$ s_{\mu}(x_1+t,\ldots,x_e+t) = \sum_{i=0}^{|\lambda|} s_{\mu}^{(i)}(x_1,\ldots,x_e) t^i\text{ for all }t\in \mathbb R.$$
So if $x_1,\ldots,x_e$ are the Chern roots of a bundle $E$ on $X$ we have characteristic classes
$$s_{\mu}^{(i)}(E)\in H^{|\mu|-i,|\mu|-i}(X;\mathbb R).$$

\begin{theorem}[= Theorem \ref{thm:logconcaveschur}]\label{thm:intrologschur}
Let $X$ be a projective manifold of dimension $d$, let $E$ be an ample bundle on $X$  with $e=\rank(E)\ge d$ and let $h$ be an ample class on $X$.   Also let $0\le \mu_N\le \cdots \le \mu_1\le e$ be a partition of $e$.

 Then the map
$$ i\mapsto \int_X s_{\mu}^{(e-i)}(E) h^{d-i} \text{ for } i=0,\ldots, d$$
is strictly log-concave.
\end{theorem}
We remark when $\mu$ is the partition given by  $\{\mu_1=e=\rank(E)\}$ 
then $s_{\mu}^{(e-i)}(E)  = c_{i}(E)$, and Theorem \ref{thm:intrologschur} becomes Theorem \ref{thm:intrologconvex}.

\begin{center}
*
\end{center}

The Hodge-Riemann bilinear relations  we have discussed above turn out to be closely related to an elementary piece of linear algebra.  Let $V$ be a complex vector space of dimension $d$ and fix a lattice $U$ in $V$.  Write $T = \Hom_{\mathbb C}(V,\mathbb C)$ and let $T^{p,q}  = \Lambda^p T \otimes \Lambda^q \overline{T}$ be the space of $(p,q)$ forms on $V$.   
Then $T^{p,p}$ is the space of sesquilinear forms on $\Lambda^p V$. 
By a \emph{K\"ahler form} $\omega$ on $V$ we mean a real strictly positive element of $T^{1,1}$.  We say that  $\omega$  is {\em rational} if its corresponding alternating skew-symmetric form on the underlying real vector space $V_\R$ of $V$ takes values in $\Q$, on $U\times U$  (see Section \ref{sec:Dihn} for further definitions and conventions).

\begin{corollary}[= Corollary \ref{cor:linearalgebra1}]\label{cor:linearalgebra1:intro}
Let $\omega_1,\ldots,\omega_e$ be rational K\"ahler forms on $V$ and  let $(\lambda,e,d)$ be in the same range as required by Theorem \ref{thm:intromain}.   Then the Schur form 
$$s_{\lambda}(\omega_1,\ldots,\omega_e)$$ has the Hodge-Riemann property. In particular the linear map
$$T^{1,1}\to T^{d-1,d-1}, \ \eta\mapsto \eta\wedge s_{\lambda}(\omega_1,\ldots,\omega_e),$$ 
is invertible.
\end{corollary}


The idea of the proof is to consider a suitable torus quotient $X$ of $V$ chosen so that $H^q(X,\Omega^p) \simeq T^{p,q}$.  We use the assumption that each $\omega_i$ is rational to find an ample vector bundle $E$ on $X$ such that $s_{\lambda}(E) = s_{\lambda}(\omega_1,\ldots,\omega_e)$ (up to scaling by a positive number).  Then Theorem \ref{thm:intromain} applied to $E$ gives Corollary \ref{cor:linearalgebra1:intro}.

We conjecture that Corollary \ref{cor:linearalgebra1:intro} continues to hold if we relax the hypothesis that the $\omega_i$ be rational,  but note that the technique used in the above proof fails as there is no longer a natural ample vector bundle $E$.  Nevertheless we have in this direction the following partial result:


\begin{proposition}[= Proposition \ref{prop:Segre}, Corollary \ref{cor:Segre}]\label{prop:Segre:intro}
Let $\omega_1,\omega_2$ be K\"ahler forms on $V$.   Then
$$\omega_1^{d-2} + \omega_1^{d-3}\wedge \omega_2 + \cdots + \omega_2^{d-2} \in T^{d-2,d-2}$$
has the Hodge-Riemann property.
\end{proposition}

Both Corollary \ref{cor:linearalgebra1:intro} and Proposition \ref{prop:Segre:intro} are elementary statements in linear algebra.  However the only proofs we are aware of are the ones given here that rely, ultimately, on Hodge-Theory.\\

\noindent {\bf Comparison with other work: }    In his work exposing a deep connection between K\"ahler geometry and convexity, Gromov \cite{Gromov90} initiated the investigation into whether there are other classes that have the Hard Lefschetz property, and proved that this is the case for certain products of (possibly different) K\"ahler classes.   This has since 
been taken up by Cattani \cite{Cattani}  and Dinh-Nguy\^en \cite{Dinh2}, \cite[Corollary 1.2]{Dinh}.   In particular  \cite{Dinh2,Dinh} explores the connection between the Hodge-Riemann property for cohomology classes and the kind of  linear algebra statements discussed above.   In \cite{Dinh2,Dinh} the authors moreover show that also lower degree products of K\"ahler classes enjoy the Hodge-Riemann property. However as we show in Example \ref{ex: higher codimension} this is no longer true in general for Schur classes of ample vector bundles. This is why we restrict in this paper to Schur classes of degree $d-2$.

For higher rank bundles the only existing statement along these lines that we are aware of is the Bloch-Gieseker Theorem \cite{BlochGieseker} which deals only with the Hard Lefschetz property (see Remark \ref{rmk:HRandLgeneralremarks}).  It is interesting to observe that both the aforementioned work of Gromov (at least in the rational case) and that of Bloch-Gieseker can be thought of as dealing with the class $c_e(E)$ for some vector bundle $E$.  
We appear to be the first to extend
 this to general Schur classes.

Ampleness of vector bundles goes back to Hartshorne \cite{Hartshorne_AmpleVectorBundles}, and analogous metric properties to Griffiths \cite{Griffiths}.  Both positivitity properties of these notions, as well as the relation between the two, has been much studied (e.g.\ \cite{Barenbaum,Berndtsson,Geertsen,Hartshorne_amplebundlesoncurves, Ionescu,Kleiman, KefengLiu,Mourougane,Snow}).   The paper that inspired the main result in this paper concerning Schur classes, as well as parts of its proof, is that of Fulton-Lazarsfeld \cite{FultonLazarsfeld}.

We refer the reader to \cite[Sec.\ 1.6]{Lazbook1} for an account of the various Hodge-Index type inequalities that can be deduced from Hodge-Theory, which takes from various sources including \cite{Demailly_numerical,Ping,Luo}. 
 Generalisations of these inequalities can be found in recent work of Xiao \cite{Xiao2,Xiao1} and Collins \cite{Collins} who approach this from the framework of concave elliptic equations. Of particular relevance to this paper are the inequalities of Khovanskii \cite{Khovanski} and Teisser \cite{TeissierBonnesen}.  \\


\noindent {\bf Main ideas in the proofs: }  We start by considering the Schur class $c_{d-2}(E)$ in the case that $\rank(E) = d-2$.  Then the Hard Lefschetz property follows from the Bloch-Gieseker Theorem.  In fact, this continues to hold if $E$ is replaced by the ample $\mathbb R$-twisted bundle $E\langle th\rangle$ where $h$ is a given ample class and $t\ge 0$.  Thus the signature of the intersection form defined by $c_{d-2}(E\langle th\rangle)$ is independent of $t$, and so a simple continuity argument implies the Hodge-Riemann bilinear relations in this case.

To deal with ample bundles of higher rank we use induction on $\rank(E)- d+2$ by applying the induction hypothesis to the product $(X\times \mathbb P^1,E\boxtimes \mathcal O_{\mathbb P^1}(1))$.  The result we want then follows from an elementary statement about quadratic forms that can be written in ``block form".  This completes the proof of Theorem \ref{thm:intromain} in the case that $s_{\lambda}(E) = c_{d-2}(E)$, and in fact gives the enhanced ``Hodge-Index" type inequality stated in Theorem \ref{thm:HI2intro}.

A similar trick gives the main step in the proof of the higher rank Khovanskii-Teissier inequalities (Theorem \ref{thm:intrologconvex}): we suppose $e=\rank(E) =d+k$, and apply the Hodge-Index inequality to the class $c_{e}(E\boxtimes \mathcal O_{\mathbb P^{k+2}}(1))$ on the product $X\times \mathbb P^{k+2}$.

To prove Theorem \ref{thm:intromain} for general Schur classes we follow the approach of Fulton-Lazarsfeld and consider intersection forms defined by suitable cone classes in ample bundles, and the effect of taking hyperplane sections on the base.  
But whereas in the original Fulton-Lazarsfeld argument the trivial observation that a positive linear combination of positive classes remains positive could be used, the analogous statement is not necessarily true of intersection forms that have the Hodge-Riemann property.  Instead we use an interplay between the Hodge-Riemann property and the enhanced Hodge-Index inequality discussed above (see \S\ref{sec:setupforproof} for a more detailed outline of this proof). \\

\noindent {\bf Organization: } Preliminaries in  \S\ref{sec:preliminaries} start with some basic statements about bilinear forms, including the aforementioned elementary, but crucial, statement about certain bilinear forms in block-form.  We also define precisely the Hodge-Riemann and Hard Lefschetz property for cohomology classes and summarize the theory of $\mathbb R$-twisted bundles.  

In  \S\ref{sec:HRcd-2} we prove Theorem \ref{thm:intromain} in the case $s_{\lambda}(E) = c_{d-2}(E)$ first when $E$ has rank $d-2$ and then for all rank.  The main result is in  \S\ref{sec:cone_classes} in which we state, and then prove, a general theorem about the Hodge-Riemann bilinear relations for intersection forms defined by cone classes.  This is applied  in \S\ref{sec:Schur} which gives details on the connection between Schur classes and cone classes (which uses standard intersection theory, as contained in \cite{FultonIT}).

In \S\ref{sec:conesofcycles} we apply this to explore  the cone of nef cycles on the self-product of a very general principally polarized abelian surface, and in \S\ref{sec:logconcavity} we apply it to prove Theorem \ref{thm:intrologschur} concerning the higher rank Khovanskii-Tessier inequalities.

In \S\ref{sec:Dihn} we turn to the K\"ahler setting and the Hodge-Riemann property for Schur classes of a collection of not necessarily rational K\"ahler forms.   Finally in \S\ref{sec:questions} we discuss a number of open questions and possible extensions.\\

\noindent {\bf Acknowledgements:} We particularly want to thank Brian Lehmann for conversations arising from an earlier version of this work, and acknowledge that the application in \S\ref{sec:conesofcycles} to the cone of cycles was suggested by him.   We also thank Izzet Coskun, Lionel Darondeau, Lawrence Ein, Christophe Mourougane, Eric Riedl and Kevin Tucker for discussions related to this work.   The first author is supported  by NSF grants DMS-1707661 and DMS-1749447.

\section{Preliminaries}\label{sec:preliminaries}

\subsection{Notation and conventions}
Our complex manifolds are assumed to be connected and vector bundles on them assumed to be holomorphic.  Given a vector bundle $E$ we denote by $\mathbb P(E)$ the space of one dimensional quotients of $E$, and by $\mathbb P_{\sub}(E)$ the space of one dimensional subspaces of $E$.  If $a,b$ are differential forms (or cohomology classes) we write $ab$ for the wedge product (resp.\ cap product) to ease notation when convenient.  A \emph{K\"ahler class} on a compact complex manifold is a strictly positive class in $H^{1,1}(X,\mathbb R)$ and an \emph{ample class} is a strictly positive class in $H^{1,1}(X,\mathbb Z)$, which we will identify with the corresponding ample divisor class when no confusion is likely.   We say a vector bundle $E$ on $X$ is \emph{ample} if the hyperplane class on $\mathbb P(E)$ is ample.

\subsection{Elementary properties of quadratic forms}
We collect here some elementary facts about bilinear and quadratic forms on finite dimensional vector spaces.  In particular in Proposition \ref{prop:quadraticformonedimensionhigher} we show certain quadratic forms that can be written in block-form satisfy an inequality similar to the classical Hodge-Index inequality.  This will be the cornerstone of the arguments in the rest of the paper.

Let $V$ be a real vector space of dimension $\rho$ and 
$$Q_V:V\times V\to \mathbb R$$ be a symmetric bilinear form on $V$.    We write
$$ Q_V(v) : = Q_V(v,v) \text{ for } v\in V$$
for the associated quadratic form.

\begin{definitionlemma}[The Hodge-Riemann property]\label{deflemma:HRofquadraticform}
Suppose there exists an $h\in V$ such that $Q_V(h)>0$.   Then the following statements are equivalent, in which case we say that $Q_V$ has the \emph{Hodge-Riemann property}.
\begin{enumerate}[(1)]
\item $Q_V$ has signature $(1,\rho-1)$.
\item There exists a subspace of dimension $\rho-1$ in $V$ on which $Q_V$ is negative definite.
\item For any $h'\in V$ such that $Q_V(h')>0$, the restriction of $Q_V$ to the primitive space 
$$ V_{h'}:= \{ v\in V : Q_V(v,h')=0\}$$
 is negative definite.
\item For any $h'\in V$ such that $Q_V(h')>0$ and all $v\in V$ the \emph{Hodge-Index inequality}
\begin{equation}  Q_V(v)  Q_V(h')\le  Q_V(v,h')^2 \label{eq:hodgeindex}\end{equation}
holds, with equality iff $v$ is proportional to $h'$.
\end{enumerate}
\end{definitionlemma}
\begin{proof}
(1) $\Rightarrow$ (2) and (3) $\Rightarrow$ (1) and (4)$\Rightarrow$ (3) are immediate, and (2)$\Rightarrow$ (3) comes from Sylvester's law of inertia.    For (3)$\Rightarrow$ (4): Given $v\in V$ choose $\lambda$ so  $ Q_V(v + \lambda h',h')=0$.  By  (3), this implies $ Q_V(v+\lambda h', v +\lambda h')\le 0$ with equality iff $v + \lambda h'=0$.  Rearranging gives (4). \end{proof}

Continuing with the above notation, suppose now $\phi\in V^*$ and consider the symmetric bilinear form on 
$$W:= V\oplus \mathbb R$$
given by
$$ Q_W( v\oplus \lambda, v'\oplus \lambda') = Q_V(v,v') + \lambda \phi(v') + \lambda' \phi(v).$$
So abusing notation a little, $Q_W$ is given in block form by
$$ Q_W= \left( \begin{array}{cc} Q_V & \phi^t \\ \phi & 0\end{array}\right).$$

\begin{proposition}\label{prop:quadraticformonedimensionhigher}
Suppose that $Q_W$ has the Hodge-Riemann property (i.e $Q_W$ has signature $(1,\rho))$ and suppose there is an $h\in V$ with \begin{enumerate}[(a)]
\item $Q_W(h)= Q_V(h)>0,$
\item $\phi(h)> 0$.
\end{enumerate}
Then 

\begin{enumerate}[(i)]
\item For all $v\in V$ it holds that
\begin{equation}\label{eq:quadraticformonedimensionhigher}Q_V(v) \phi(h) \le 2  Q_W(v,h) \phi(v)\end{equation}
with equality if and only if $v=0$.  
\item $Q_V$ has the Hodge-Riemann property.  In fact  $Q_V$ is negative definite on $\ker \phi$ which has codimension 1.
\end{enumerate}
\end{proposition}
\begin{proof}
Let $v\in V$ and $v\oplus \lambda \in W$.  By the Hodge-Index inequality \eqref{eq:hodgeindex} for $Q_W$ we have
\begin{equation}\label{eq:quadraticinlambda} Q_W(v\oplus\lambda,h)^2- Q_W(v \oplus \lambda)  Q_W(h) \ge 0\end{equation}
with equality if and only if $v\oplus\lambda$ is proportional to $h$.     The idea of the proof is to think of \eqref{eq:quadraticinlambda} as a quadratic  polynomial  in $\lambda\in \mathbb R$ that is always non-negative, which by elementary algebra gives an inequality among its coefficients.

To ease notation let
 $$\begin{array}{ll}
 a := Q_W(v)= Q_V(v)  &  d: = \phi(h)\\  b:=Q_W(v,h)= Q_V(v,h) &   e:=\phi(v)\\ c: = Q_W(h) = Q_V(h) 
 \end{array}$$
 and observe that by hypothesis $c,d>0$.
Then  \eqref{eq:quadraticinlambda} becomes
 \begin{equation}(b+\lambda d)^2-c(a+ 2\lambda e)\ge 0 \text{ for all } \lambda\in \mathbb R\label{eq:quadraticinlambda:repeat}\end{equation}
 with equality if and only if $v\oplus \lambda$ is proportional to $h$.
 
Now substituting
$$\lambda_0: = \frac{ce-db}{d^2}$$
into \eqref{eq:quadraticinlambda:repeat} and simplifying yields
$$2dbe-ad^2-ce^2\ge 0.$$
So, using $c>0$, we have
\begin{equation}ad^2\le 2bde - ce^2 \le 2bde\label{eq:quadraticformonedimensionhigher2} \end{equation}
and hence
$$ad\le 2be$$
which is precisely the inequality \eqref{eq:quadraticformonedimensionhigher} we wanted to show.

Suppose now equality holds for $v$ in \eqref{eq:quadraticformonedimensionhigher}.  
In the notation above this says precisely $ad=2be$ and so \eqref{eq:quadraticformonedimensionhigher2} implies $ce^2=0$ and so $e=0$.  
 Moreover equality holds in \eqref{eq:quadraticinlambda:repeat} when $\lambda= \lambda_0$, and so $v\oplus \lambda_0$ is proportional to $h$.  In turn this implies that
 $v$ is proportional to $h$, say $v=\kappa h$ for some $\kappa\in \mathbb R$ and so $0=e=\kappa d$, 
$\kappa=0$ and hence $v=0$ as desired proving (i).

The final statements are clear, for our assumption that $\phi(h)>0$ implies that $\ker \phi$ has codimension 1, and \eqref{eq:quadraticformonedimensionhigher} implies $Q_V$ is negative definite on $\ker \phi$.  Thus (ii) holds.
\end{proof}

\subsection{The Hodge-Riemann property for cohomology classes}

Let $X$ be a compact K\"ahler manifold of dimension $d\ge 2$, $\omega_0$ be a K\"ahler class on $X$ and fix an integer $0\le k\le d$ so that $d-k$ is even.    Let $$\Omega\in H^{k,k}(X;\mathbb R)$$ and consider the intersection pairing
\begin{align*} 
Q_\Omega(\alpha,\beta):=Q(\alpha,\beta) &:= \int_X \alpha \wedge \Omega\wedge {\beta} \text{ for } \alpha,\beta\in H^{\frac{d-k}{2},\frac{d-k}{2}}(X;\mathbb R).\end{align*}
We denote by
$$H^{\frac{d-k}{2},\frac{d-k}{2}}_{p,\Omega}(X;\mathbb R)$$
the \emph{primitive cohomology} of $\Omega$, by which we mean the kernel of the map $$H^{\frac{d-k}{2},\frac{d-k}{2}}(X;\mathbb R) \to H^{\frac{d+k+2}{2} ,\frac{d+k+2}{2}}(X;\mathbb R) \text{ given by } \alpha\mapsto \Omega\wedge \omega_0\wedge \alpha.$$

\begin{definition}[Hard Lefschetz Property]
We say that $\Omega$ has the \emph{Hard Lefschetz property}  
     if the map
\begin{align}
H^{\frac{d-k}{2},\frac{d-k}{2}}(X;\mathbb R)& \to H^{\frac{d+k}{2},\frac{d+k}{2}}(X;\mathbb R)\label{eq:hardlefschetz}\\
\alpha&\mapsto \Omega\wedge \alpha\nonumber
\end{align}
is an isomorphism.  
\end{definition}

\begin{definition}[Hodge-Riemann Property]\label{def:HRmanifolds}
 We say that $\Omega$ has the \emph{Hodge-Riemann property }  (with respect to $\omega_0$) if
 \begin{enumerate}
     \item $\int_X \Omega.\omega_0^{d-k}> 0$ and
     \item $(-1)^{\frac{d-k}{2}}Q_\Omega$  is positive definite on the primitive cohomology $H^{\frac{d-k}{2},\frac{d-k}{2}}_{p,\Omega}(X;\mathbb R)$.
     \end{enumerate}
\end{definition}

\begin{remarks}\label{rmk:HRandLgeneralremarks}
\begin{enumerate}[(1)]
\item The map \eqref{eq:hardlefschetz} being an isomorphism is equivalent to $Q_{\Omega}$ being non-degenerate.  Thus the Hodge-Riemann property implies the Hard Lefschetz property.
\item When $k=d$ the Hard Lefschetz property is equivalent to $\int_X \Omega\neq 0$, and the Hodge-Riemann property is equivalent to $\int_X \Omega>0$. 
    \item If $\omega \in H^{1,1}(X;\mathbb R)$ is a K\"ahler class then the classical Hard Lefschetz Theorem (see for instance \cite[Theorem 6.4]{Voisin}) says that $\omega^k$ has both the Hard Lefschetz and Hodge-Riemann property for $k\le d$.
    \item More generally, suppose $\omega_1,\ldots,\omega_k\in H^{1,1}(X;\mathbb R)$ are K\"ahler classes and $k\le d$.  Then it is known that
    $$\Omega:= \omega_1\wedge \cdots \wedge \omega_{k}$$
has both the Hard Lefschetz and Hodge-Riemann property. This is due to Gromov \cite{Gromov90} when $k=d-2$, and in general due to Cattani \cite{Cattani} as well as Dihn-Nguyen \cite{Dinh2}, \cite[Corollary 1.2]{Dinh}  (in fact the last two citations consider more generally the corresponding statement on $(p,q)$-forms).
    \item Let $E$ be an ample vector bundle of rank $k\le d$ on $X$.  Then a Theorem of Bloch-Gieseker (to be discussed further in \ref{sec:blochgieseker}) implies that the Chern class $c_k(E)$ has the Hard Lefschetz property.  
    \item Since $\Omega\in H^{k,k}(X,\mathbb R)$ is assumed to be real,  the Hard Lefschetz property is equivalent to the map on the complex vector spaces
    \begin{align*}
H^{\frac{d-k}{2},\frac{d-k}{2}}(X)& \to H^{\frac{d+k}{2},\frac{d+k}{2}}(X)\label{eq:hardlefschetz:complex}\\
\alpha&\mapsto \Omega.\alpha\nonumber
\end{align*}
being an isomorphism.  And there is an analogous statement for the Hodge-Riemann property.    Thus there is no loss in considering real cohomology throughout, which we do for simplicity.
    \item The Hard Lefschetz and Hodge-Riemann properties are each clearly invariant under scaling $\Omega$ by a positive real number.  However neither property are closed under taking convex combinations (see Remark \ref{rem:convexcombinations}).
   
\end{enumerate}
\end{remarks}

\subsection{$\mathbb R$-twisted vector bundles}

We recall briefly the notion of $\mathbb R$-twisted bundles (essentially following \cite[Section 6.2, 8.1.A]{Lazbook2}, \cite[p457]{Miyoka}).  Let $E$ be a vector bundle of rank $e$ on a base $X$ and $\delta\in H^{1,1}(X;\mathbb R)$.   Then we can consider the so-called $\mathbb R$-twised bundle of rank $e$ denoted by
$$E\langle \delta \rangle.$$
which is to be understood as a formal object, having Chern classes defined by the rule
\begin{equation}c_p(E\langle \delta \rangle) := \sum_{k=0}^p\binom{e-k}{p-k} c_k(E) \delta^{p-k} \text{ for } 0\le p\le e.\label{eq:defcherntwisted}\end{equation}
Said another way, if $x_1,\ldots,x_e$ are the Chern roots of $E$ then $x_1+\delta,\ldots,x_e+\delta$ are the Chern roots of $E\langle \delta\rangle$.

This definition is made so that if $\delta$ is integral, so $\delta = c_1(L)$ for some line bundle $L$, then
$$c_p(E\langle \delta \rangle) = c_p(E\otimes L).$$

The twist of an $\mathbb R$-twisted vector bundle by a $\delta'\in H^{1,1}(X,\mathbb R)$ is defined by the obvious rule
$$E\langle \delta \rangle \langle \delta'\rangle : = E\langle \delta + \delta'\rangle,$$
and the tensor product of an $\mathbb R$-twisted vector bundle and a line bundle $L$ is given by the rule
$$ E\langle \delta \rangle \otimes L: = E \langle \delta  + c_1(L) \rangle.$$
Consider now the projective bundle $\pi:\mathbb P(E)\to X$ of one-dimensional quotients in $E$ with hyperplane class $h_{\mathbb P(E)}: = c_1(\mathcal O_{\mathbb P(E)}(1))$.

\begin{definition}
We say that the $\mathbb R$-twisted vector bundle $E\langle \delta\rangle$ is ample (resp.\ nef) if the class
$$ h_{\mathbb P(E)} +  \pi^* \delta \in H^{1,1}(\mathbb P(E))$$
is ample (resp.\ nef).
\end{definition}

We observe that this agrees with the usual definition when $\delta = c_1(L)$ for some line bundle $L$.  For then $\mathbb P(E) \simeq \mathbb P(E\otimes L)$ and under this identification
$$h_{\mathbb P(E\otimes L)} = h_{\mathbb P(E)} +  \pi^* \delta,$$
so $E\langle c_1(L)\rangle$ is ample if and only if $E\otimes L$ is ample.

Now on $\mathbb P_{\sub}(E)$ we have a tautological quotient bundle $U$ of rank one less than $E$,  which fits into the tautological sequence
$$ 0 \to K \to \pi^* E \to U\to0.$$
For the twisted case we identify $\mathbb P_{\sub}(E\langle \delta \rangle)$ with $\mathbb P_{\sub}(E)$ and the tautological bundle on the former is defined to be 
$$ U \langle \pi^* \delta \rangle$$
which fits into the twisted exact sequence
\begin{equation} 0 \to K\langle \pi^*\delta \rangle \to \pi^* E\langle \delta\rangle \to U\langle \pi^* \delta \rangle \to 0.\label{eq:universalquotienttwisted}\end{equation}

\subsection{Schur polynomials}

By a partition $\mu$ of an integer $e$ we mean a sequence $0\le \mu_N\le \cdots \le \mu_1$ such that $|\mu|: = \sum_i \mu = e$.  Given such a partition one has the Schur polynomial $s_{\mu}(x_1,\ldots,x_e)$, which is symmetric (we will need almost nothing about the theory of such polynomials, but the interested reader will find many accounts e.g. \cite{FultonYoung}).  

 When $x_1,\ldots,x_e$ are the chern roots of an $\mathbb R$-twisted bundle $E$ on $X$ we thus have a well-defined class $$s_{\mu}(E) \in H^{|\mu|,|\mu|}(X;\mathbb R).$$
We will have use for the following ``derived" Schur polynomials (compare \cite[Theorem 1.5]{JangSooKim}).

\begin{definition}\label{def:lowerschur}
Let $\mu$ be a partition.   For each $0\le i\le |\mu|$ let $s_{\mu}^{(i)}(x_1,\ldots,x_e)$ be defined by requiring that
$$ s_{\mu}(x_1+t,\ldots,x_e+t) = \sum_{i=0}^{|\mu|} s_{\mu}^{(i)}(x_1,\ldots,x_e) t^i \text{ for all } t\in \mathbb R.$$
\end{definition}
Clearly then $s_{\mu}^{(i)}$ is a symmetric polynomial of degree $|\mu|-i$ and  $s_{\mu}^{(0)} = s_{\mu}$.      A formal calculation, that is left to the reader, implies
\begin{equation}s_{\mu}^{(i)}(x_1+t,\ldots,x_e+t) = \sum_{k=i}^{|\mu|} \binom{k}{i} s_{\mu}^{(k)}(x_1,\ldots,x_e) t^{k-i}\label{eq:lowershurexpansion}.\end{equation}
Once again, thinking of $x_1,\ldots,x_e$ are the Chern roots of an $\mathbb R$-twisted bundle $E$ on $X$ gives a well-defined characteristic class
$$s_{\mu}^{(i)}(E)\in H^{|\mu|-i,|\mu|-i}(X;\mathbb R).$$ 
Moreover if $\delta\in H^{1,1}(X;\mathbb R)$ then, by definition,
$$ s_{\mu}(E\langle \delta\rangle) = \sum_{i=0}^{|\mu|} s_{\mu}^{(i)}(E) \delta^i,$$
and \eqref{eq:lowershurexpansion} implies
\begin{equation}s_{\mu}^{(i)}(E\langle \delta\rangle) = \sum_{k=i}^{|\mu|}\binom{k}{i} s_{\mu}^{(k)}(E) \delta^{k-i}.\label{eq:lowercherntwistexpansion}\end{equation}

\begin{example}[Chern classes]\label{ex: Chern classes}
Consider the simplest partition of $e$ consisting of just one integer $\mu_1=e$, at which point $s_{\mu}(x_1,\ldots,x_e) = x_1\cdots x_e$.  So if $E$ is an $\mathbb R$-twisted vector bundle of rank $e$ then $s_{\mu}(E) = c_{e}(E)$, and moreover 
$$s_{\mu}^{(i)}(E) = c_{e-i}(E)\text{ for all } 0\le i\le e.$$
Then \eqref{eq:lowercherntwistexpansion} rearranges to become
\begin{equation} c_p (E\langle \delta\rangle) = \sum_{k=0}^p \binom{e-k}{p-k} c_k(E) \delta^{p-k} \text{ for } 0\le p\le e,\label{eq:chernclasstensorproduct:full}\end{equation}
which agrees with \eqref{eq:defcherntwisted} (as it must).  We record for later use that in particular if $1\le p\le e$ and $t\in \mathbb R$ then
\begin{equation}c_p(E\langle t\delta \rangle) = c_p(E) +t(e-p+1)c_{p-1}(E)\delta+ O(t^2).\label{eq:chernclasstensorproduct}\end{equation}
\end{example}

\begin{example}[Segre classes]
At the other extreme we may consider the partition $(1)^e=(1,\ldots,1)$ of length $e$.  Then $s_{\mu}(E) = (-1)^e s_e(E)$ where $s_e(E)$ is the Segre class.    Letting $e= \rank (E)$ we have  \cite[3.1.1]{FultonIT}
$$s_{(1)^e}(E\otimes L) = \sum_{j=0}^e \binom{2e-1}{2e-1-j} s_{(1)^{e-j}}(E) c_1(L)^{j}.$$
and thus
$$ s_{(1)^e}^{(i)}(E) = \binom{2e-1}{2e-1-i}s_{(1)^{e-i}}(E).$$
\end{example}

\begin{example}[Derived Schur polynomials of Low degree]
For convenience of the reader we list some of the derived Schur classes of low degree for a bundle $E$ of rank $e$

$$s_{(1)} = c_1, \quad s_{(1)}^{(1)} = e \text{ for  } e\ge 1.$$
$$s_{(2,0)} = c_2, \quad s_{(2,0)}^{(1)} = (e-1) c_1 \quad s_{(2,0)}^{(2)}=\binom{e}{2} \text{ for } e \ge 2.$$
$$s_{(1,1)} = c_1^2 -c_2, \quad s_{(1,1)}^{(1)}=(e+1)c_1, \quad s_{(1,1)}^{(2)}= \binom{e+1}{2} \text{ for } e \ge 2.$$
$$s_{(3,0,0)} = c_3, \quad s_{(3,0,0)}^{(1)} = (e-2)c_2, \quad s_{(3,0,0)}^{(2)}= \binom{e-1}{2} c_1, \quad s_{(3,0,0)}^{(3)}=\binom{e}{3} \text{ for } e \ge 3.$$
$$s_{(2,1,0)}= c_1c_2 -c_3,\quad  s_{(2,1,0)}^{(1)}= 2c_2 + (e-1)c_1^2, \quad s_{(2,1,0)}^{(2)}=(e^2-1) c_1, $$
$$\quad  s_{(2,1,0)}^{(3)}=2\binom{e+1}{3} \text{ for } e \ge 3.$$
$$s_{(1,1,1)}= c_1^3 -2c_1c_2 + c_3, \quad s_{(1,1,1)}^{(1)}=(e+2) (c_1^2-c_2) ,\quad s_{(1,1,1)}^{(2)}= \binom{e+2}{2} c_1,$$
$$s_{(1,1,1)}^{(3)}= \binom{e+2}{3},\text{ for } e \ge 3.$$
\end{example}

\subsection{The Bloch-Gieseker theorems}\label{sec:blochgieseker}

\begin{theorem}[Bloch-Gieseker I]\label{thm:blochgiesekerI}
Let $X$ be projective smooth of dimension $d$ and $E$ be an $\mathbb R$-twisted ample vector bundle of rank $e$ on $X$.  Let $s = \min\{e,d\}$ and assume $i\le (d-s)/2$.  Then the map
$$H^{i,i}(X;\mathbb R) \to H^{i+s,i+s}(X;\mathbb R) \quad \alpha\mapsto \alpha\wedge c_s(E)$$
is injective.
\end{theorem}
\begin{proof}
This originates in \cite{BlochGieseker} (see also \cite[7.1.10]{Lazbook2}).  We observe that \cite{BlochGieseker} is not stated for $\mathbb R$-twists, but the proof goes through essentially unchanged (see  \cite[p113]{Lazbook2}, \cite[Proposition 2.1]{DemaillyPeternellSchneider}).
\end{proof}

\begin{theorem}[Bloch-Gieseker II]\label{thm:BlochGiesekerII}
Let $X$ be projective smooth of dimension $d$ and $E$ be an $\mathbb R$-twisted ample vector bundle  of rank $e$ on $X$ with $e\ge d$.  Then $\int_X c_d(E)>0$.
\end{theorem}
\begin{proof}
See \cite[Proposition 2.2]{BlochGieseker} or \cite[Corollary 8.2.2]{Lazbook2}.
\end{proof}

We collect some simple consequences of this result.

\begin{corollary}\label{cor:nonzerochernclass}
Let $X$ be projective smooth of dimension $d$ and $E$ be a rank $e$ $\mathbb R$-twisted ample vector bundle and $h\in H^{1,1}(X,\mathbb Z)$ be an integral ample class. 
Then 
$$\int_X c_q(E) h^{d-q} >0 \quad \text{ for all } q\le\min\{d,e\}.$$
\end{corollary}
\begin{proof}
Fix $q\le \min\{d,e\}$.  Without loss of generality we may assume $h$ is very ample.  Then the class $h^{d-q}$ is represented by a smooth subvariety $Y\subset X$ of dimension $q$.   Now $E|_Y$ is an ample $\mathbb R$-twisted bundle of rank $e\ge q$, so by Theorem \ref{thm:BlochGiesekerII}
$$0<\int_Y c_{q}(E)= \int_X c_{q} (E) h^{d-q}$$
as required.
\end{proof}


\begin{corollary}\label{cor:cn2nondegen}
Let $X$ have dimension $d\ge 2$ and $E$ be a $\mathbb R$-twisted ample and of rank $e=d-2$.   Then the intersection form $$Q(\alpha,\alpha') = \int_X \alpha c_{d-2}(E) \alpha \text{ for } \alpha,\alpha'\in H^{1,1}(X;\mathbb R)$$  is non-degenerate
\end{corollary}
\begin{proof}
Suppose $Q(\alpha,\beta)=0$ for all $\beta\in H^{1,1}(X;\mathbb R)$.   Then by Serre duality,  $\alpha c_{d-2}(E)=0$, and so Theorem \ref{thm:blochgiesekerI} yields $\alpha=0$.
\end{proof}

\section{The Hodge-Riemann property for $c_{d-2}(E)$}\label{sec:HRcd-2}
\subsection{The case $\rank(E) = d-2$}


\begin{proposition}\label{prop:cn2I}
Let $E$ be an ample $\mathbb R$-twisted bundle of rank $d-2$ on a projective manifold $X$ of dimension $d\ge 2$.  Then $c_{d-2}(E)$ has the Hodge-Riemann property with respect to any ample class $h$ on $X$.  
\end{proposition}
\begin{proof}
By a consequence of the Bloch-Gieseker Theorem for ample $\mathbb R$-twisted vector bundles (Corollary \ref{cor:cn2nondegen}), for all $t\ge 0$  the intersection form 
$$Q_{t}(\alpha):= \int_X \alpha c_{d-2} (E\langle th \rangle) \alpha \text{ for } \alpha\in H^{1,1}(X;\mathbb R)$$ is non-degenerate.   Now for small $t$ we have $$ c_{d-2}(E\langle th\rangle) = t^{d-2} h^{d-2} + O(t^{d-3}).$$
Observe that for an intersection form $Q$, having signature $(1,h^{1,1}(X)-1)$ is invariant under multiplying $Q$ by a positive multiple, and is an  open condition as $Q$ varies continuously.  Thus since we know that $h^{d-2}$ has the Hodge-Riemann property, the intersection form $(\alpha,\beta)\mapsto \int_X \alpha h^{d-2} \beta$ has signature $(1,h^{1,1}(X)-1)$, and hence so does $Q_t$ for $t$ sufficiently large.   But $Q_t$ is non-degenerate for all $t\ge 0$, and hence $Q_t$ must have this same signature for all $t\ge 0$.

Next recall from Corollary \ref{cor:nonzerochernclass} that $\int_X c_{d-2}(E) h^2 >0$.  Thus  $c_{d-2}(E)$ has the Hodge-Riemann property with respect to $h$ as claimed.
\end{proof}

\subsection{The case $\rank(E) \ge d-1$}\label{sec:rankhigher}


\begin{theorem}\label{thm:HI2}
Let $X$ be a projective manifold of dimension $d\ge 2$ and $h$ be an ample class on $X$.  Suppose $E$ is an ample $\mathbb R$-twisted vector bundle of rank $e\ge d-1$ on $X$.    Then

\begin{enumerate}[(1)]
\item  For all $\alpha\in H^{1,1}(X;\mathbb R)$ it holds that 
   \begin{equation}\label{eq:hodgeindexiii}\int_X \alpha^2 c_{d-2}(E) \int_X h c_{d-1}(E) \le 2 \int_X \alpha h c_{d-2}(E) \int_X \alpha c_{d-1}(E)\end{equation}
  with equality  if and only if $\alpha=0$. 
  
\item  The class $c_{d-2}(E)$ has the Hodge-Riemann property with respect to $h$.   In fact if 
$$ W: = \{ \alpha \in H^{1,1}(X;\mathbb R) : \int_X \alpha c_{d-1}(E) = 0\}$$
then $\dim W = h^{1,1}(X)-1$ and the intersection form $$Q(\alpha,\alpha') = \int_X \alpha c_{d-2}(E) \alpha'\text{ for } \alpha,\alpha'\in H^{1,1}(X;\mathbb R)$$ is negative definite on $W$. 
\end{enumerate}

  \end{theorem}

\begin{proof}

Consider the following two statements that depend on a  given $j\ge 0$ \bigskip

($P_j$) For any projective manifold $X'$ of dimension $d'\ge 2$, any ample class $h'$ on $X'$, and any ample $\mathbb R$-twisted vector bundle $E'$ on $X'$ with $\rank(E') =d'-2+j$ the class $c_{d'-2}(E')$ has the Hodge-Riemann property with respect to $h'$.
\bigskip

($Q_j$)  For any projective manifold $X'$ of dimension $d'\ge 2$, any ample class $h'$ on $X'$ and any ample $\mathbb R$-twisted vector bundle $E'$ on $X'$ with $\rank(E') =d'-2+j$ and any $\alpha\in H^{1,1}(X';\R)$ it holds that
$$\int_{X'} \alpha^2 c_{d'-2}(E') \int_{X'} h' c_{d'-1}(E') \le 2 \int_{X'} \alpha h' c_{d'-2}(E')\int_{X'}\alpha c_{d'-1}(E')$$
with equality if and only if $\alpha=0$.\bigskip

Then statement $(P_0)$ holds, as this is  the content of Proposition \ref{prop:cn2I}.   We will show that 
\begin{enumerate}[(a)]
\item $(Q_j)\Rightarrow (P_j) \text{ for all } j\ge 1,$
\item $(P_{j-1})\Rightarrow (Q_{j}) \text{ for all } j\ge 1.$
\end{enumerate}


Clearly these together imply that $(Q_j)$ holds for all $j\ge 1$ which is precisely statement (1) of the Theorem.\bigskip

Proof of (a): Let $j\ge 1$ and assume that $(Q_j)$ holds.  Let $X'$ be a projective manifold of dimension $d'$ and $E'$ be an $\mathbb R$-twisted ample vector bundle with $\rank(E') = d'-2+j$ and $h'$ be an ample class on $X'$.   Then since $(Q_j)$ is assumed to hold,
the quadratic form
\begin{equation}(\alpha,\alpha') \mapsto \int_{X'} \alpha c_{d'-2}(E')\alpha' \text{ for } \alpha,\alpha'\in H^{1,1}(X';\mathbb R)\label{eq:quadraticform55}\end{equation}
is negative definite on the space
$$ W': = \{ \alpha \in H^{1,1}(X';\mathbb R) : \int_{X'} \alpha c_{d'-1}(E') = 0 \}.$$
But ampleness of $E'$ implies (Corollary \ref{cor:nonzerochernclass}) that $h'\notin W'$, and so $W'$ has codimension 1 in $H^{1,1}(X';
\mathbb R)$.  Thus the  quadratic form in \eqref{eq:quadraticform55} has signature $(1,h^{1,1}(X')-1)$ and so $c_{d'-2}(E)$ has the Hodge-Riemann property.  Hence $(P_j)$ holds and we have proved (a).  Observe that in doing so we have also proved that item (1) in the Theorem implies item (2).\bigskip

Proof of (b): Suppose $j\ge 1$ and $(P_{j-1})$ holds and we want to show $(Q_j)$.  To this end let $X$ be a projective manifold of dimension $d$ and $h$ be an ample class on $X$ and $E$ be an ample $\mathbb R$-twisted vector bundle on $X$ with $\rank(E)=:e:=d-2+j$.      We have to show that
for any $\alpha\in H^{1,1}(X;\R)$ it holds that
\begin{equation}\label{eq:Qjrepeat}\int_{X} \alpha^2 c_{d-2}(E) \int_{X} h c_{d-1}(E) \le 2 \int_{X} \alpha h c_{d-2}(E)\int_{X}\alpha c_{d-1}(E)\end{equation}
with equality if and only if $\alpha=0$.

Set
$X' := X\times \mathbb P^1$ and ${E'} : = E\boxtimes \mathcal O_{\mathbb P^1}(1)$ which is an ample $\mathbb R$-twisted bundle.    Observe that $d':= \dim(X') = d+1$ and  
$$\rank({E}') =\rank(E) = d-2+j = (d+1) - 2 + (j-1)= d' -2 + (j-1).$$
Hence by the assumption $(P_{j-1})$ we know that $c_{d'-2}(E')$ has the Hodge-Riemann property.
Write $\tau: = c_1(\mathcal O_{\mathbb P^1}(1))$ and
$$ H^{1,1}(X';\mathbb R) = H^{1,1}(X;\mathbb R) \oplus H^{1,1}(\mathbb P^1;\R) = H^{1,1}(X;\mathbb R) \oplus \mathbb R\langle \tau \rangle.$$
Observe $1\le d'-2=d-1\le d-2+j=e$ and moreover $e-(d'-2)+1 = e-d+2=j$.  So using the identity for the Chern class of a tensor product \eqref{eq:chernclasstensorproduct} and the fact that $\tau^2=0$ we get
$$c_{d'-2}(E') = c_{d-1}(E') = c_{d-1}(E) + j c_{d-2}(E) \tau.$$
Now define
$$\phi:H^{1,1}(X;\mathbb R)\to \mathbb R \text{ by } \phi(\alpha): = \int_{X} \alpha c_{d-1}(E),$$
and
$$\mathcal Q_X(\alpha,\alpha') = j\int_X \alpha c_{d-2}(E) \alpha' \text{ for } \alpha,\alpha'\in H^{1,1}(X;\mathbb R).$$
Then
\begin{align*} \mathcal Q_{X\times \mathbb P^1}(\alpha \oplus \lambda \tau, \alpha'\oplus \lambda'\tau) &:=\int_{X\times \mathbb P^1} (\alpha + \lambda \tau) c_{d'-2}(E') (\alpha' + \lambda'\tau) \\&=  \mathcal Q_X(\alpha,\alpha') + \lambda\phi(\alpha') + \lambda' \phi(\alpha)\end{align*}
which as we have already observed has the Hodge-Riemann property.   Finally notice that as $E$ is ample we have $\mathcal Q_{X}(h)>0$ and $\phi(h)>0$.    Thus we are in precisely the setup of Proposition \ref{prop:quadraticformonedimensionhigher}  giving
$$\mathcal Q_X(\alpha) \phi(h) \le 2 \mathcal Q_X(\alpha,h) \phi(\alpha)$$
with equality if and only if $\alpha=0$, which yields \eqref{eq:Qjrepeat}.  Hence $(Q_j)$ holds and the proof of (b) is complete.
\end{proof}

\begin{corollary}\label{cor:cn-2HR}
Suppose that $E$ is an ample $\mathbb R$-twisted vector bundle on a projective manifold $X$ of dimension $d$ and $\rank(E)\ge d-2$.  Then $c_{d-2}(E)$ has the Hodge-Riemann property.  In particular for all $\alpha\in H^{1,1}(X;\mathbb R)$ we have
\begin{equation}\int_X \alpha^2 c_{d-2}(E) \int_X h^2 c_{d-2}(E) \le \left(\int_X \alpha c_{d-2}(E) h \right)^2\label{eq:hodgeindexcn-2definitive}\end{equation}
with equality if and only if $\alpha$ is proportional to $h$.
\end{corollary}
\begin{proof}
This is Proposition \ref{prop:cn2I} when $\rank(E) =d-2$ and Theorem \ref{thm:HI2} when $\rank(E)\ge d-1$.
\end{proof}

The results proved in this section will be essential in our proof of the Hodge-Riemann property for Schur classes.   In fact, what we will need is that both the above Hodge-Index inequality and the more general inequality  \eqref{eq:hodgeindexiii} continue to hold if $E$ and $h$ are merely nef on a base that is irreducible but not necessarily smooth. 

\begin{corollary}\label{cor:inequalitiesinnefcase}
Let $P$ be a smooth projective variety, and $h$ be a nef class on $P$.  Suppose that $C\subset P$ is irreducible of dimension $n$ and that $E$ is a nef $\mathbb R$-twisted bundle on $P$.   For $\alpha,\alpha'\in H^{1,1}(P;\mathbb R)$ set
$$\mathcal Q(\alpha,\alpha'): = \int_{C} \alpha c_{n-2}(E) \alpha'$$
$$\phi(\alpha) := \int_C \alpha c_{n-1}(E).$$
Then for all $\alpha\in H^{1,1}(P;\mathbb R)$ we have
\begin{equation}\mathcal Q(\alpha) \mathcal Q(h) \le \mathcal Q(\alpha,h)^2\label{eq:HInef}\end{equation}
and
\begin{equation}\mathcal Q(\alpha) \phi(h) \le 2 \mathcal Q(\alpha,h) \phi(\alpha).\label{eq:HIplusnef}\end{equation}
(We emphasise that we are making no claims here as to what happens when equality holds in \eqref{eq:HInef} or \eqref{eq:HIplusnef}).
\end{corollary}
\begin{proof}
Suppose first that $C=P$  (so in particular $C$ is smooth).  If $\rank(E)<n-2$ then $\mathcal Q$ is identically zero and there is nothing to prove.  So we may assume $\rank(E)\ge n-2$.  Let $\eta$ be an ample class on $P$.  Then for any $t>0$ the bundle $E\langle t\eta\rangle$ is ample and the class $h_{t}: = h + t\eta$ is ample.  Now set
$$\mathcal R_t(\alpha,\alpha'): = \int_C \alpha c_{n-2}(E\langle t \eta\rangle) \alpha'$$
$$\phi_t(\alpha) : = \int_C \alpha c_{n-1}(E\langle t \eta\rangle).$$
Then we have from Proposition \ref{prop:cn2I} and Theorem \ref{thm:HI2} respectively that for all $\alpha\in H^{1,1}(P;\mathbb R)$ it holds that
\begin{equation}\mathcal R_t(\alpha) \mathcal R_t(h_t) \le \mathcal R_t(\alpha,h_t)^2\label{eq:HInef:repeate}\end{equation}
and
\begin{equation}\mathcal R_t(\alpha) \phi_t(h_t) \le 2 \mathcal R_t(\alpha,h_t) \phi_t(\alpha)\label{eq:HIplusnef:repeat}\end{equation}
(observe that the latter inequality holds trivially if $\rank(E) = n-2$ for then $\phi_t=0$, and otherwise Theorem \ref{thm:HI2}  applies).   Letting $t\to 0$ gives \eqref{eq:HInef} and \eqref{eq:HIplusnef} which completes the proof when $C$ is smooth.

Now suppose that $C$ is irreducible of dimension $n$ inside $P$ as in the statement of the theorem.  Let $\pi:C'
\to C$ be a resolution of singularities.  We denote the induced morphism $C'\to P$ also by $\pi$, so there is a pullback map
$$\pi^*: H^{1,1}(P;\mathbb R) \to H^{1,1}(C';\mathbb R).$$
Observe that $E':= \pi^*E$ and $h':= \pi^* h$ are nef on $C'$.  So by the previous paragraph the result we want applies for the triple $(C',E',h')$.   
Now for any $\alpha,\alpha'\in H^{1,1}(P;\mathbb R)$ we have 
$\int_{C} c_{n-2}(E) \alpha \alpha' = \int_{C'} c_{n-2}(\pi^*E) (\pi^*\alpha) (\pi^*\alpha')$ and  also $$\int_C c_{n-1}(E) \alpha = \int_{C'} c_{n-1}(E')\pi^* \alpha.$$  Hence the result for $C$ follows from that for $C'$.
\end{proof}

\section{The Hodge-Riemann property for cone classes}\label{sec:cone_classes}


\subsection{Statement}

Let $X$ be smooth, projective of dimension $d\ge 4$ and  $h\in H^{1,1}(X;\mathbb R)$ be very ample.  Let $\pi:F\to X$ be an ample $\mathbb R$-twisted vector bundle on $X$ of rank $f+1$.
Consider
$$ \pi:P : = \mathbb P_{\sub}(F)\to X$$
and denote by $U$ the universal quotient
$$0\to K\to \pi^* F\to U \to 0$$
 so $U$ is a nef $\mathbb R$-bundle (recall the universal quotient in the $\mathbb R$-twisted case was discussed in \eqref{eq:universalquotienttwisted}).   Suppose $C\subset P$ is a 
  subvariety of codimension $d-2$ that is flat over $X$ with irreducible fibers (in fact in the case of interest $C$ will be locally a product). The main result of this section is the following:

\begin{theorem}\label{thm:maincone}
Assume $f\ge d$ and set $n=\dim C$.  Then for $2\le i\le d$ the bilinear form
$$(\alpha,\alpha') \mapsto\int_{C} (\pi^* \alpha) c_{n-2-(d-i)}(U) (\pi^* \alpha')  (\pi^* h)^{d-i}  \text{ for } \alpha,\alpha' \in H^{1,1}(X;\mathbb R)$$
has the Hodge-Riemann property (i.e. it has signature $(1,h^{1,1}(X)-1))$. 
\end{theorem}

\subsection{Setup for the proof}\label{sec:setupforproof}

Since 
$$0\to K\to \pi^* F\to U \to 0$$
we have $U$ is a quotient of the nef 
 $\mathbb R$-twisted  bundle $\pi^*F$ and hence $U$ is nef of rank $f$. 
Set
$$\zeta : =- c_1(K)$$ 
which is relatively ample over $X$ (but note we do not claim any further positivity of $\zeta$).
Then
$$H^{1,1}(P;\mathbb R) = \pi^* H^{1,1}(X;\mathbb R) \oplus \mathbb R\zeta.$$
We have
$$ n   = \dim C = f+ d-(d-2) = f+2$$
so our hypothesis $f\ge d$ implies
\begin{equation}n\ge d+2.\label{eq:assumptionnd}\end{equation}
For convenience set
\begin{equation}n_i:= n-d+i\label{eq:boundni}.\end{equation}
and observe that by \eqref{eq:assumptionnd},
\begin{equation} i+2\le n_i\le n \text{ for } 1\le i\le d.\label{eq:inequalityni}\end{equation}

\begin{definition}
Given a $\mathbb R$-twisted vector bundle $U'$ on $P$ and $1\le i\le d$ define a bilinear form on $H^{1,1}(P;\mathbb R)$  by
 \begin{align*}
\mathcal Q_i(\beta,\beta';U'):= \int_{C} \beta c_{n-d+i-2}(U') (\pi^*h)^{d-i} \beta'\end{align*}
for $\beta,\beta'\in H^{1,1}(P;\mathbb R)$.  We also set 
 $$\mathcal F_i(\beta;U') := \mathcal Q_i(\beta;U') \mathcal Q_i(\pi^* h;U') - \left(\mathcal Q_i(\beta, \pi^*h;U')\right)^2.$$
 When $U'$ is taken to be the universal quotient on $P$ we write these as
 $$\mathcal Q_i(\beta,\beta'): = \mathcal Q_i(\beta,\beta';U)$$
 and
 $$\mathcal F_i(\beta) : = \mathcal F_i(\beta;U).$$
\end{definition}

\begin{theorem}[Fulton-Lazarsfeld]\
It holds that
\begin{equation}\mathcal Q_i(\pi^* h) > 0 \text{ for } 2\le i\le d.\label{eq:thmfl}\end{equation}
\end{theorem}
\begin{proof}
We observe here that we are using ampleness of $F$.  The statement \eqref{eq:thmfl} is that
$$\int_C c_{n-d+i-2} (U) (\pi^*h)^{d-i+2}> 0 \text{ for } 2\le i \le d,$$
which is \cite[Theorem 2.3]{FultonLazarsfeld} (we observe that in the cited work the quantity $a_S$ is given by $\dim C - \dim  \pi(C) = \dim C - \dim X= n-d$ since we are assuming $C$ is flat over $X$).  We remark also that in \cite[0.2]{FultonLazarsfeld} the authors specify that by $\mathbb P(F)$ they mean the projective bundle of one-dimensional subspaces of $F$.  
\end{proof}

\begin{definitionlemma}\label{definitionlemmaAi}Let $2 \le i \le d$.  We say $(A_i)$ holds if any of the following equivalent conditions are true:
\begin{enumerate}[(1)]
\item For $\beta\in H^{1,1}(P;\mathbb R)$,
$$ \mathcal F_i(\beta) =0  \text{ implies }\beta =\kappa \pi^*h \text{ for some } \kappa \in \mathbb R.$$
\item For $\beta\in H^{1,1}(P;\mathbb R)$,
$$\mathcal Q_i(\beta,\pi^*h)=0 \text{ and } \mathcal Q_i(\beta)=0 \text{ imply } \beta=0.$$
\item The quadratic form $\mathcal Q_i$ has the Hodge-Riemann property (i.e. it has signature $(1,h^{1,1}(P)-1)$.
\end{enumerate}
\end{definitionlemma}

That these are equivalent is a consequence of the following:
\begin{lemma}\label{lem:Q'nefFnonpositive}
Assume $U'$ is a nef $\mathbb R$-twisted vector bundle on $P$.  Then for all $2\le i\le d$ it holds that
\begin{equation} \mathcal F_i(\beta;U')\le 0 \text{ for all } \beta\in H^{1,1}(P;\mathbb R).\label{eq:mathcalFnonpositive}\end{equation}
\end{lemma}
\begin{proof}
Fix $2\le i\le d$.   Then $h^{d-i}$ is represented by a smooth $Y\subset X$ of dimension $i$.  Let $C':= \pi^{-1}(Y)\cap C$ which has dimension $n-d+i=:n_i$ and 
$$\mathcal Q_i(\beta,\beta')  = \int_{C'} \beta c_{n_i-2}(U') \beta',$$
Since $C$ is assumed to be flat over $X$ with irreducible fibers  (in fact locally a product over $X$ with irreducible fiber) we have that $C'$ is irreducible, and clearly projective.  Moreover $\pi^* h$ is clearly nef on $C'$.    Hence the result we want is implied by the analysis we did in the previous section (specifically Corollary \ref{cor:inequalitiesinnefcase}).
\end{proof}


\begin{proof}[Proof of Definition-Lemma \ref{definitionlemmaAi}]
We have from \eqref{eq:thmfl} that $\mathcal Q_i(\pi^*h)>0$.  Combined with Lemma \ref{lem:Q'nefFnonpositive}, the claimed equivalence between these statements is the elementary statement about bilinear forms given in Lemma \ref{deflemma:HRofquadraticform}.
\end{proof}

We next make a similar definition that captures the stronger inequality that was considered in Section \ref{sec:rankhigher}.

\begin{definition}
Suppose $1\le i\le d-1$ and $U'$ is a $\mathbb R$-twisted vector bundle on $P$.   For $\beta\in H^{1,1}(P;\mathbb R)$ set
\begin{align*}
\phi_i(\beta; U')&:=\mathcal Q_{i+1}(\pi^*h,\beta;U') \\
&= \int_C c_{n-d+i-1}(U') (\pi^* h)^{d-i-1} (\pi^*h) \beta \\
&= \int_C c_{n-d+i-1}(U') (\pi^* h)^{d-i}  \beta. \\
\end{align*}
So $\phi_i(\cdot;U')$ lies in the dual space of $H^{1,1}(P;\mathbb R)$.    Moreover define
$$ \mathcal G_i(\beta;U'):= \mathcal Q_i(\beta;U') \phi_i(\pi^* h;U') - 2 \mathcal Q_i(\beta,\pi^*h;U') \phi_i(\beta;U').$$
When $U'$ is the universal quotient bundle $U$ we write
$$\phi_i(\beta) :  = \phi_i(\beta;U)$$
$$\mathcal G_i(\beta): = \mathcal G_i(\beta;U).$$
\end{definition}

\begin{lemma}\label{lem:mathcalGnonpositive}
Assume $U'$ is a nef $\mathbb R$-twisted bundle on $P$.  Then for all $2\le i\le d-1$ it holds that
$$ \mathcal G_i(\beta;U')\le 0 \text{ for all } \beta\in H^{1,1}(P;\mathbb R)$$
\end{lemma}
\begin{proof}
The proof is precisely the same as that of Lemma \ref{lem:Q'nefFnonpositive} since, with the notation in that proof, 
$$\phi_i(\beta) = \int_{C'} \beta c_{n_i-1}(U).$$
\end{proof}

\begin{definition}
Let $2\le i\le d-1$.  We say $(B_i)$ holds if for any $\beta\in H^{1,1}(P;\mathbb R)$
$$ \mathcal G_i(\beta) =0   \Rightarrow \beta=0.$$
\end{definition}

\begin{remark}
Since $\rank(U)=n-2$ we clearly have $c_{n-1-d+d}(U) = c_{n-1}(U)=0$ and hence (extending the above notation appropriately) $\phi_{d}=0$ and $\mathcal G_d\equiv 0$.  For this reason we only consider $\mathcal G_i$ and property $(B_i)$ when $2\le i\le d-1$.
\end{remark}

We can now break the steps of the proof of Theorem \ref{thm:maincone} as separate propositions, that will each be proved in turn in the next subsections.

\begin{proposition}\label{prop:BiimpliesAi+1}
Suppose $(B_i)$ holds for some $2\le i\le d-1$.  Then $(A_{i+1})$ holds.
\end{proposition}

\begin{proposition}\label{prop:ChaspropertyA2}
$(A_2)$ holds
\end{proposition}

\begin{proposition}\label{prop:AiimpliesBi}
Suppose $(A_i)$ holds for some $2\le i\le d-2$.  Then $(B_i)$ holds.
\end{proposition}

\begin{proposition}\label{prop:finalstep}
Suppose $(A_{d-2})$ holds.  Then the the restriction of $\mathcal Q_d$ to the subspace $\pi^*H^{1,1}(X;\mathbb R)\subset H^{1,1}(P)$ has the Hodge-Riemann property.  
\end{proposition}

\begin{proof}[Proof of Theorem \ref{thm:maincone}]
Combining Propositions \ref{prop:ChaspropertyA2},\ref{prop:BiimpliesAi+1} and \ref{prop:AiimpliesBi} and induction on $i$ gives that $(A_i)$ holds for $2\le i\le d-1$.  Thus $\mathcal Q_i$ for $2\le i\le d-1$ has the Hodge-Riemann property over $H^{1,1}(P;\mathbb R)$ and since $\mathcal Q_i(h)>0$ this implies it also has the Hodge-Riemann property over $\pi^* H^{1,1}(X;\mathbb R) \subset H^{1,1}(P;\mathbb R)$.    This proves the claim for $2\le i\le d-1$.

Moreover, we have $(A_{d-2})$ holds, so  Proposition \ref{prop:finalstep} applies giving the required statement when $i=d$.

\end{proof}

\begin{remark}
It is worth observing also that $\mathcal Q_d$ does not generally have the Hodge-Riemann property over all of $H^{1,1}(P)$.   For, as we will see in \eqref{eq:chernhigherU}, since $n\ge d+2$, 
$$c_{n-2}(U) \zeta = c_{n-1}(U)=0$$
where the last equality follows as $\rank(U)=n-2$, and so
$$\mathcal Q_d(\zeta,\beta) = \int_{C} c_{n-2}(U) \zeta \beta = 0 \text{ for all } \beta\in H^{1,1}(P).$$
In particular $\mathcal Q_d$ is degenerate, so cannot have the Hodge-Riemann property.
\end{remark}

\subsection{Proof of Proposition \ref{prop:BiimpliesAi+1}}


\begin{lemma}\label{lem:Omegaiexpansion}
For any $2\le i\le d$ 
\begin{equation} \mathcal Q_i(\beta;U\langle t\pi^*h\rangle) = \mathcal Q_i(\beta;U) + t (d-i+1) \mathcal Q_{i-1}(\beta; U) + O(t^2).\label{eq:Omegaiexpansion}
\end{equation}
\end{lemma}
\begin{proof}
Since $n\ge d+2$ we have $1\le n-d+i-2\le n-2 = \rank(U)$.  Thus using the equation for the Chern class of the tensor product \eqref{eq:chernclasstensorproduct}, and observing that $\rank(U)-(n-d+i-2)+1 = d-i+1$, gives
$$ c_{n-d+i-2}(U\langle t\pi^*h\rangle) = c_{n-d+i-2}(U) + t(d-i+1) c_{n-d+i-3}(U) (\pi^*h) + O(t^2)$$
Multiplying this by $\beta^2$ and integrating over $C$ gives \eqref{eq:Omegaiexpansion}.
\end{proof}

\begin{lemma}\label{lem:equalitycaseofF}
Fix $1\le i\le d-1$. Let $\beta\in H^{1,1}(P;\R)$ be such that
\begin{equation}\label{eq:equalitycaseofFhypothesis}
\mathcal Q_{i+1} (\beta,\pi^*h) = 0 = \mathcal Q_{i+1}(\beta).
\end{equation}
Then
\begin{align}
\mathcal Q_{i}(\beta) &=0\label{eq:propBimpliesAi+1conclusion2}\\
\mathcal G_{i}(\beta)&=0.\label{eq:propBimpliesAi+1conclusion3}\end{align}
\end{lemma}
\begin{proof}

Observe first \eqref{eq:equalitycaseofFhypothesis} clearly implies 
\begin{equation}\label{eq:propBiimpliesAi+1setup2}\mathcal F_{i+1}(\beta;U)=0.\end{equation}  On the other hand for $t\in \mathbb R$ with $|t|$ sufficiently small the $\mathbb R$-twisted bundle $F\langle th \rangle$ remains ample.  Thus the $\mathbb R$-twisted bundle $U\langle t\pi^* h\rangle$ is nef, and so by Lemma \ref{lem:Q'nefFnonpositive}
\begin{equation}\label{eq:propBiimpliesAi+1setup3} f(t):=\mathcal F_{i+1}(\beta; U\langle t\pi^* h \rangle) \le 0 \text{ for all } |t|\ll 1.
\end{equation}
So \eqref{eq:propBiimpliesAi+1setup2} says $f(0)=0$, which together with \eqref{eq:propBiimpliesAi+1setup3} implies
$$\frac{df}{dt}|_{t=0} =0.$$
We may calculate this derivative using Lemma \ref{lem:Omegaiexpansion}.  In fact up to terms of order $O(t^2)$,

\begin{align*}
f(t) &= [\mathcal Q_{i+1}(\beta) + t(d-i)\mathcal Q_{i}(\beta)][\mathcal Q_{i+1}(\pi^*h) + t(d-i) \mathcal Q_{i}(\pi^*h)] \\
&\ \ \ - \left[\mathcal Q_{i+1}(\beta,\pi^*h) + t(d-i)\mathcal Q_{i}(\beta,\pi^*h)\right]^2+ O(t^2)\\
& = t(d-i)\mathcal Q_{i}(\beta) \mathcal Q_{i+1}(\pi^*h) + O(t^2)
\end{align*}
where the last equality uses  our assumption \eqref{eq:equalitycaseofFhypothesis}.  
Hence 
\begin{align*}
0=  (d-i) \mathcal Q_i(\beta) \mathcal Q_{i+1}(\pi^*h).
\end{align*}
Now recall \eqref{eq:thmfl} gives $\mathcal Q_{i+1}(\pi^*h)>0$.
Hence $\mathcal Q_{i}(\beta)=0$ which is  \eqref{eq:propBimpliesAi+1conclusion2}.

Finally $\phi_i(\beta) = \mathcal Q_{i+1}(\beta,\pi^*h)=0$  by hypothesis, and hence
 $$\mathcal G_i(\beta)= \mathcal Q_i(\beta) \phi_i(\pi^*h) - 2 \mathcal Q_i(\beta,\pi^*h)\phi_i(\beta)=0$$ as claimed in \eqref{eq:propBimpliesAi+1conclusion3}.
\end{proof}

\begin{proof}[Proof of Proposition \ref{prop:BiimpliesAi+1}]
Fix $2\le i\le d-1$ and suppose $(B_i)$ holds, and the aim is to show $(A_{i+1})$ holds.  To this end suppose $\beta\in H^{1,1}(P;\R)$ satisfies
\begin{equation}\mathcal Q_{i+1}(\beta) =0= \mathcal Q_{i+1}(\beta, \pi^* h)\end{equation}
Then Lemma \ref{lem:equalitycaseofF} implies
$$\mathcal G_i(\beta)=0.$$
 But by  $(B_i)$ this implies $\beta=0$.   Looking back at Definition-Lemma \ref{definitionlemmaAi} we conclude $(A_{i+1})$ holds as desired.
\end{proof}

\subsection{Proof of Proposition \ref{prop:ChaspropertyA2}}

\begin{lemma}\label{lem:assumption:chernclass}
For all $p$ we have
\begin{equation}c_p(U) = c_{p-1}(U) \zeta + \pi^* c_{p}(F)\label{eq:chernidentityU}\end{equation}

In particular if $\Omega\in H^{j,j}(X;\mathbb R)$ then
\begin{align} c_p(U)(\pi^* \Omega) &= c_{p-1}(U) \zeta (\pi^* \Omega) \text{ for } j+p\ge d+1\label{eq:chernhigherU}\\
 c_p(U)(\pi^* \Omega)& = c_{p-2}(U) \zeta^2 (\pi^* \Omega) \text{ for } j+p\ge d+2.\label{eq:chernhigherUdouble}
 \end{align}
\end{lemma}

\begin{proof}
The first equation follows from the exact sequence
$$0\to  K\to \pi^* F\to U \to 0$$
and $\zeta = -c_1(K)$ so 
$$(1-\zeta) c(U) = c(K) c(U) =  \pi^*c(F)$$
and thus taking the degree $p$ part, 
$$c_p(U) - \zeta c_{p-1}(U) = \pi^*c_p(F).$$
Equation \eqref{eq:chernhigherU} follows as $\dim X=d$ so if $\Omega\in H^{j,j}(X;\mathbb R)$ and $j+p\ge d+1$ then $c_p(F).\Omega=0$.   The proof of \eqref{eq:chernhigherUdouble}  follows from two applications of \eqref{eq:chernhigherU}.
\end{proof}

We proceed now to show $(A_2)$  holds.   To this end suppose $\beta\in H^{1,1}(P;\mathbb R)$ satisfies
\begin{align}
\mathcal Q_2(\beta)&=0\label{eq:A2setup1} \text{ and }\\
\mathcal Q_2(\beta,\pi^*h)&=0\label{eq:A2setup2}
\end{align}
Our aim is to show that $\beta=0$.

We have $\beta = \pi^* \alpha + \lambda \zeta$ for some $\alpha\in H^{1,1}(X;\mathbb R)$ and $\lambda\in \mathbb R$.  Then 
\begin{align*}
0&=\mathcal Q_1(\beta) \tag{by Lemma \ref{lem:equalitycaseofF}}\\
&=\int_{C} (\pi^*h)^{d-1} c_{n-d-1}(U) \beta^2\\
&=\int_{C} (\pi^*h)^{d-1} c_{n-d-1}(U) (\pi^*\alpha + \lambda\zeta)^2\\
&=2\lambda \int_{C} (\pi^*h)^{d-1} c_{n-d-1}(U)\pi^* \alpha \zeta + \lambda^2 \int_{C} (\pi^*h)^{d-1} c_{n-d-1}(U)  \zeta^2 \tag{since  $\alpha^2h^{d-1}=0$ as $\dim X=d$}\\
&=2\lambda \int_{C} (\pi^*h)^{d-1} c_{n-d}(U)\pi^* \alpha  + \lambda^2 \int_{C} (\pi^*h)^{d-1} c_{n-d+1}(U) \tag{by \eqref{eq:chernhigherU} and \eqref{eq:chernhigherUdouble} using also $n\ge d+2$}\\
&= 2\lambda A + \lambda^2 B
\end{align*} 
where
\begin{align*}
A&: = \int_{C} (\pi^*h)^{d-1} c_{n-d}(U)\pi^* \alpha\\
B&:= \int_{C} (\pi^*h)^{d-1} c_{n-d+1}(U)\\
&=  \int_{C} (\pi^*h)^{d-3} (\pi^*h)^2 c_{n-d+1}(U)\\
&=\mathcal Q_3(\pi^*h)\\
&>0. \tag{by \eqref{eq:thmfl}}
\end{align*}
On the other hand,
\begin{align*}
0 &= \mathcal Q_2(\beta,\pi^*h) \tag{ by \eqref{eq:A2setup2}}\\
&=\int_{C} (\pi^*h)^{d-2} c_{n-d}(U) \beta (\pi^*h)\\
&=\int_{C} (\pi^*h)^{d-1} c_{n-d} (U)(\pi^*\alpha + \lambda \zeta) \\
&= A + \lambda B \tag{ by \eqref{eq:chernhigherU} and $n\ge d+2$}
\end{align*}

Thus in summary we have $  2\lambda A+ \lambda^2 B =0= A + \lambda B$ and $B\neq 0$ which forces $\lambda=0$.   

Let $W$ be the class of the fibre of $C$ (as we are assuming $C$ is locally a product, the class of this fibre is the same for every fibre).  Then as $\lambda=0$,
$$ 0=A = \int_C c_{n-d}(U) (\pi^* h)^{d-1} \pi^* \alpha = \int_W c_{n-d}(U) \int_X h^{d-1} \alpha.$$
But $\int_W c_{n-d}(U) = \int_W \zeta^{n-d}>0$ as $\zeta$ is relatively ample. Therefore
\begin{equation} \int_X h^{d-1} \alpha=0 \label{eq:primitive}\end{equation}
Furthermore
\begin{align*}
0 &= \mathcal Q_2(\beta) \tag{from \eqref{eq:A2setup1}} \\
&= \mathcal Q_2(\pi^* \alpha) \tag{as $\lambda=0$}\\
&= \int_C c_{n-d}(U) (\pi^*h)^{d-2} (\pi^*\alpha)^2 \\
&= \int_W c_{n-d}(U)\int_X h^{d-2}\alpha^2\\
& = \int_W \zeta^{n-d} \int_X h^{d-2}\alpha^2
\end{align*}
and thus
$$\int_X h^{d-2}\alpha^2=0.$$
Coupled with equation \ref{eq:primitive} this implies $\alpha=0$ by the Hodge-Riemann bilinear relations for $h$.
This completes the proof that $(A_2)$ holds.

\subsection{Proof of Proposition \ref{prop:AiimpliesBi}}

Fix $2\le i\le d-2$, suppose $(A_i)$ holds, and the aim is to show $(B_i)$ holds.    To this end, suppose that $\beta\in H^{1,1}(P;\mathbb R)$ satisfies 
\begin{equation}
\mathcal G_i(\beta)=0\label{eq:proofAiimpliesBiassumption}
\end{equation}
We have to show that $\beta=0$.

\begin{claim}
We have
\begin{equation}\phi_i(\beta)=0\label{eq:lemmaAiimpliesBi1}\end{equation}
\begin{equation}\mathcal Q_i(\beta)=0\label{eq:lemmaAiimpliesBi2}
\end{equation}
and
\begin{equation} \mathcal Q_i(\beta,\beta') \phi_i(\pi^*h) = \mathcal Q_i(\beta,\pi^*h) \phi_i(\beta')
\label{eq:lemmaAiimpliesBi3}\end{equation}
for all $\beta'\in H^{1,1}(P;\mathbb R)$.
\end{claim}
\begin{proof}
Let $\beta'\in H^{1,1}(P;\mathbb R)$.  Then by Lemma \ref{lem:mathcalGnonpositive}
$$ g(t):=\mathcal G_i(\beta + t\beta')\le 0 \text{ for all } t\in \mathbb R.$$
Moreover \eqref{eq:proofAiimpliesBiassumption} implies $g(0)=0$ and so
$$ \frac{dg}{dt}|_{t=0} =0.$$
Now ignoring terms of order $O(t^2)$,
\begin{align*}
\mathcal G_i( \beta + t\beta') &= \mathcal Q_i(\beta + t\beta') \phi_i(\pi^*h) - 2\mathcal Q_i(\beta + t\beta',\pi^*h) \phi_i( \beta + t\beta')\\
&=(\mathcal Q_i(\beta) + 2t\mathcal Q_i(\beta, \beta'))\phi_i(\pi^* h) \\
&\ \ \ \ \ - 2( \mathcal Q_i(\beta,\pi^*h) + t \mathcal Q_i(\beta',\pi^*h))(\phi_i(\beta) + t\phi_i(\beta')) +O(t^2)\\
&= \mathcal G_i(\beta) + 2t \mathcal Q_i(\beta,\beta') \phi_i(\pi^* h) - 2t \mathcal Q_i(\beta',\pi^*h) \phi_i(\beta) \\&- 2t\mathcal Q_i(\beta,\pi^*h)\phi_i(\beta') + O(t^2).
\end{align*}
Hence
\begin{equation}\label{eq:runoutofnames} 0 =  \mathcal Q_i(\beta,\beta') \phi_i(\pi^* h) -  \mathcal Q_i(\beta',\pi^*h) \phi_i(\beta) - \mathcal Q_i(\beta,\pi^*h) \phi_i(\beta').
\end{equation}
In particular this applies when $\beta' = \pi^*h$ at which point the first and third terms cancel giving
$$ 0 = \mathcal Q_i(\pi^*h) \phi_i(\beta)$$
and since $\mathcal Q_i(\pi^*h)>0$ \eqref{eq:thmfl} this yields 
$$\phi_i(\beta)=0$$
giving \eqref{eq:lemmaAiimpliesBi1}.    In turn this implies
$$ 0 = \mathcal G_i(\beta) = \mathcal Q_i( \beta) \phi_i(\pi^*h)$$
and  $\phi_i(\pi^*h) =\mathcal Q_{i+1}(\pi^*h)>0$  giving \eqref{eq:lemmaAiimpliesBi2}.  Finally \eqref{eq:runoutofnames} also yields
$$ \mathcal Q_i( \beta,\beta') \phi_i(\pi^*h) = \mathcal Q_i(\beta,\pi^*h) \phi_i(\beta')$$
for all $\beta'\in H^{1,1}(P;\mathbb R)$ which is 
\eqref{eq:lemmaAiimpliesBi3}.
\end{proof}

Now by our assumption that $(A_i$) holds, the quadratic form $\mathcal Q_i$ has the Hodge-Riemann property.  In particular it is non-degenerate.  Hence there is a $\gamma\in H^{1,1}(P;\mathbb R)$ dual to $\phi_i$,  i.e.\ such that
$$\mathcal Q_i(\beta', \gamma) = \phi_i(\beta') \text{ for all } \beta'\in H^{1,1}(P;\mathbb R).$$
We observe that since $\phi_i(\pi^*h)>0$ we have $\phi_i\neq 0$ and hence $\gamma\neq 0$.

\begin{claim}\label{claim:alphagamma}
There exists a $\kappa\in \mathbb R$ such that $\beta = \kappa\gamma$
\end{claim}
\begin{proof}
From \eqref{eq:lemmaAiimpliesBi3} with $\gamma$ substituted for $\beta'$,
$$\mathcal Q_i(\beta,\pi^*h) \phi_i(\gamma)=\mathcal Q_i( \beta,\gamma) \phi_i(\pi^*h) = \phi_i(\beta) \phi_i(\pi^*h) =0
$$
where the last equality comes from \eqref{eq:lemmaAiimpliesBi1}.  

Suppose first that $\mathcal Q_i(\beta, \pi^*h)=0$.  Recall we already know from \eqref{eq:lemmaAiimpliesBi2} that $\mathcal Q_i(\beta)=0$ and $\mathcal Q_i(\pi^*h)>0$.  Thus since $\mathcal Q_i$ has the Hodge-Riemann property we deduce that $\beta=0$ so the Claim certainly holds with $\kappa=0$.

So we may assume $\mathcal Q_i(\beta,\pi^*h)\neq 0$, and so 
$$\phi_i(\gamma)=0.$$
Thus, in summary, the classes $\beta$ and $\gamma$ both lie in $\ker(\phi_i)$ and also in the null cone of $\mathcal Q_i$.  Recall $\mathcal Q_i$ has signature $(1,h^{1,1}(P)-1)$ and is negative semidefinite on $\ker(\phi_i)$ by Lemma \ref{lem:mathcalGnonpositive}.  But this is only possible if $\beta$ is proportional to $\gamma$ (this is a formal statement about such bilinear forms that for completeness we include in Lemma \ref{lem:intersectionnullandcodim1}).  This finishes the proof.
\end{proof}

\begin{lemma}\label{lem:intersectionnullandcodim1}
Let $\mathcal Q$ be a bilinear form on a finite dimensional vector space $V$ with the Hodge-Riemann property.  Let $W\subset V$ be a subspace of codimension $1$ on which $\mathcal Q$ is negative semidefinite.    Then if $\beta,\gamma\in W$ satisfy $\mathcal Q(\beta) = \mathcal Q(\gamma)=0$ and $\gamma\neq 0$ then $\beta= \kappa \gamma$ for some $\kappa\in \mathbb R$.
\end{lemma}
\begin{proof}
Let $h\in V$ be such that $\mathcal Q(h)>0$.  For $t\in \mathbb R$  we have $\beta+t\gamma\in W$ and hence
$$0 \ge \mathcal Q(\beta + t\gamma) = 2t\mathcal Q(\beta,\gamma).$$
Since this holds for all $t$ we conclude $\mathcal Q(\beta,\gamma)=0$.    Thus we actually have
$$ 0 = \mathcal Q (\beta + t\gamma) \text{ for all } t\in \mathbb R.$$
If $\mathcal Q(\gamma,h)=0$ then as $\mathcal Q(\gamma)=0$ and $\mathcal Q$ has the Hodge-Riemann property we would have $\gamma=0$ which is absurd.  So $\mathcal Q(\gamma,h)\neq 0$.  Thus we may find $t_0$ so $\mathcal Q(\beta + t_0\gamma,h)=0$.  Since also $\mathcal Q(\beta + t_0\gamma) =0$ we deduce from the Hodge-Riemann property of $\mathcal Q$ that $\beta+t_0\gamma=0$ and we are done.
\end{proof}

\begin{proof}[Completion of proof of Proposition \ref{prop:AiimpliesBi}]
Suppose for contradiction $\beta\neq 0$.    Invoking Claim \ref{claim:alphagamma} we may rescale $\beta$ and assume without loss of generality that actually $\beta = \gamma$, i.e.
$$\mathcal Q_i(\beta,\beta') = \phi_i(\beta') \text{ for all } \beta'\in H^{1,1}(P;\mathbb R).$$
In particular 
$$\mathcal Q_i(\beta,\zeta) = \phi_i(\zeta).$$
Now
\begin{align*}\mathcal Q_i(\beta,\zeta)& = \int_{C} \beta c_{n_i-2}(U) \zeta (\pi^*h)^{d-i}\\ 
&=\int_{C} c_{n_i-1}(U) \beta (\pi^*h)^{d-i} \tag{from \eqref{eq:chernhigherU} since $n_i-1 +d-i = n-1\ge d+1$ by \eqref{eq:inequalityni}}\\
&= \phi_i(\beta)\\
&=0\tag{ by \eqref{eq:lemmaAiimpliesBi1}}
\end{align*}
but
\begin{align*}
\phi_i(\zeta) &= \int_{C} c_{n_i-1}(U)\zeta (\pi^*h)^{d-i}\\
&= \int_{C} c_{n_i}(U) (\pi^*h)^{d-i}\tag{from \eqref{eq:chernhigherU}}\\
&=\int_{C} c_{n-d+i}(U) (\pi^*h)^{d-i}\\
&=\int_{C} c_{n-d+j-2}(U) (\pi^*h)^{d-j+2} \tag{where  $j:=i+2$}\\
&= \mathcal Q_j(\pi^*h)\\
&>0 \tag{by \eqref{eq:thmfl} as $2\le j\le d$}
\end{align*}
which is absurd.  Hence we must actually have $\beta=0$ and the proof of Proposition \ref{prop:AiimpliesBi} is complete.
\end{proof}
\subsection{Proof of Proposition \ref{prop:finalstep}}

Assume $(A_{d-2})$ holds.  Suppose $\alpha\in H^{1,1}(X;\mathbb R)$ is such that
\begin{equation}\mathcal Q_{d}(\pi^*\alpha) =0= \mathcal Q_{d}(\pi^*\alpha,\pi^*h).\label{eq:laststepsetup}\end{equation}
We have to show that $\alpha=0$.  To this end, we apply Lemma \ref{lem:equalitycaseofF} to get
\begin{align}\mathcal Q_{d-1}(\pi^* \alpha) &= 0\label{eq:casedstep1a}\\
 \mathcal G_{d-1}(\pi^*\alpha)&=0.\label{eq:casedstep1b}
 \end{align}
Now consider
$$ g(t) := \mathcal G_{d-1}(\pi^*\alpha; U\langle t\pi^*h\rangle)$$ 
so by the above  $g(0)=0$.  On the other hand $U\langle t\pi^*h \rangle$ is nef for $|t|\ll 1$, so Lemma \ref{lem:mathcalGnonpositive} implies $g(t)\le 0$ for all $|t|\ll 1$.  Hence
\begin{equation}\frac{dg}{dt}|_{t=0} =0.\label{eq:diffgiszero}\end{equation}
\begin{lemma}\label{lem:derivativeg}
We have
\begin{equation}\mathcal Q_{d-2}(\pi^* \alpha) \mathcal Q_{d-2}(\zeta) = \mathcal Q_{d-2}(\pi^* \alpha,\zeta)^2\label{eq:derivativeg}\end{equation}
\end{lemma}
\begin{proof}
We need an elementary computation of the derivative of $g$.    First we have
\begin{align*}
\mathcal Q_{d-1}(\pi^*\alpha; U\langle t\pi^*h\rangle) &= \mathcal Q_{d-1}(\pi^*\alpha) + 2t \mathcal Q_{d-2}(\pi^*\alpha)+ O(t^2) \tag{ by Lemma \ref{lem:Omegaiexpansion}} \\
&= 2t \mathcal Q_{d-2}(\pi^*\alpha)+ O(t^2) \tag{ by \eqref{eq:casedstep1a}}\\
\mathcal Q_{d-1}(\pi^*\alpha,\pi^*h; U\langle t\pi^*h\rangle) &= \mathcal Q_{d-1}(\pi^*\alpha,\pi^*h) + O(t)\\
\phi_{d-1}(\pi^* \alpha; U\langle t\pi^*h\rangle) &=  \mathcal Q_{d}(\pi^*\alpha,\pi^*h; U\langle t\pi^*h\rangle)\tag{ by definition of $\phi_{d-1}$}\\
&=\mathcal Q_{d}(\pi^*\alpha,\pi^*h) +  t \mathcal Q_{d-1}(\pi^*\alpha,\pi^*h) +O(t^2)  \tag{ by Lemma \ref{lem:Omegaiexpansion}}\\
&=  t \mathcal Q_{d-1}(\pi^*\alpha,\pi^*h) +O(t^2)  \tag{ by \eqref{eq:laststepsetup}}\\
\phi_{d-1}(\pi^*h; U\langle t\pi^*h\rangle)&=
\mathcal Q_{d}(\pi^*h; U\langle t\pi^*h\rangle)\tag{ by definition of $\phi_{d-1}$}\\
&=\mathcal Q_{d}(\pi^*h) +O(t)  \tag{ by Lemma \ref{lem:Omegaiexpansion}}
\end{align*}
So
\begin{align*}
g(t)=\mathcal G_{d-1}(\pi^*\alpha; U\langle t\pi^*h \rangle) &= 
\mathcal Q_{d-1} (\pi^*\alpha; U\langle t\pi^*h \rangle)  
\phi_{d-1}(\pi^*h;U\langle t\pi^*h \rangle) \\
&\ \ \ \ - 2\mathcal Q_{d-1} (\pi^*\alpha,\pi^*h; U\langle t\pi^*h \rangle)  \phi_{d-1}(\pi^*\alpha;U\langle t\pi^*h \rangle)\\
& = 2t\mathcal Q_{d-2}(\pi^*\alpha) \mathcal Q_d(\pi^*h) -2t\mathcal Q_{d-1}(\pi^*\alpha,\pi^*h)^2 + O(t^2)
\end{align*}
Thus \eqref{eq:diffgiszero} implies
\begin{equation} \mathcal Q_{d-2}(\pi^* \alpha) \mathcal Q_d(\pi^*h ) = \mathcal Q_{d-1}(\pi^*\alpha,\pi^*h)^2.\label{eq:conclusiondiffg}\end{equation}
We manipulate this as follows:
\begin{align}
\mathcal Q_d(\pi^* h) &= \int_{C} c_{n-2}(U) (\pi^*h)^2 \\
&=\int_{C} c_{n-4}(U)\zeta^2  (\pi^*h)^2  \tag{by \eqref{eq:chernhigherUdouble}}\nonumber\\
&= \mathcal Q_{d-2}(\zeta)\label{eq:conclusiondiffg2}
\end{align}
and
\begin{align}
\mathcal Q_{d-1}(\pi^*\alpha,\pi^*h) & = \int_C c_{n-3}(U)(\pi^*\alpha)(\pi^*h)^2\nonumber \\
&= \int_C c_{n-4}(U) \zeta (\pi^*\alpha) (\pi^* h)^2 \tag {by \eqref{eq:chernhigherU}}\nonumber \\
& = \mathcal Q_{d-2}(\zeta,\pi^*\alpha).\label{eq:conclusiondiffg3}
\end{align}
Combining \eqref{eq:conclusiondiffg} and \eqref{eq:conclusiondiffg2} and \eqref{eq:conclusiondiffg3} gives \eqref{eq:derivativeg}.
\end{proof}

\begin{proof}[Completion of proof of Proposition \ref{prop:finalstep}]
Observe that a futher consequence of \eqref{eq:conclusiondiffg2} is that
$$\mathcal Q_{d-2}(\zeta) = \mathcal Q_d(\pi^*h) >0$$
where the last inequality uses \eqref{eq:thmfl}.

Now our assumption that $(A_{d-2})$ holds means that $\mathcal Q_{d-2}$ has the Hodge-Riemann property.    Thus the Hodge-Index inequality (Definition-Lemma \ref{deflemma:HRofquadraticform}(4)) yields
$$\mathcal Q_{d-2}(\beta) \mathcal Q_{d-2}(\zeta) \le \mathcal Q_{d-2}(\beta,\zeta)^2 \text{ for all } \beta\in H^{1,1}(P;\mathbb R)$$
with equality if and only if $\beta$ is proportional to $\zeta$.  

But \eqref{eq:derivativeg} says precisely that equality holds when $\beta$ is replaced by $\pi^*\alpha$, and thus we must have that $\pi^* \alpha$ is proportional to $\zeta$.  But this is only possible if $\pi^*\alpha=0$ which implies $\alpha=0$ completing the proof.
\end{proof}

\section{The Hodge-Riemann Property for Schur Classes}\label{sec:Schur}

\subsection{Schur classes}\label{subsec:Schur}

We next apply the main result of the previous section to certain cone classes that recover the Schur classes of our ample vector bundle.  The first part of this material is standard, and can mostly be found in \cite{Lazbook2}, and entirely in \cite{FultonIT}.   For completeness we show how this works.

Let $X$ be projective of dimension $d\ge 4$ and $E$ be a vector bundle on $X$ of rank 
$$e := \rank(E) \ge 2.$$
Let $$0\le \lambda_N\le \lambda_{N-1}\le \cdots \le \lambda_1$$
be a partition of  $b$ with $\lambda_1\le e$, $1\le b\le N$ and $b\le d$.  (For our purposes $N$ may be taken to be $b$, but we prefer to look at the more general situation $N\ge b$.)
In particular
$$|\lambda|:=\sum_{i=1}^N \lambda_i = b$$
and
$$0\le \lambda_i\le\min(e, b) \text{ for all  } i=1,\ldots,N.$$ 
Set
$$ a_i:= e+i-\lambda_i.$$
Then 
$a_1=e+1-\lambda_1\ge 1,$
$ a_{i+1} = e + (i+1) - \lambda_{i+1} \ge e + i - \lambda_i + 1 = a_i+1,$
$ a_i = e + i - \lambda_i\ge e+i - e \ge i$
and
$ a_N = e+N-\lambda_N \le e+N.$
Fix a real vector space $V$ of dimension 
$$\dim V = e+N.$$
The above inequalities say we may fix a nested subsequence $A$ of subspaces
$$ 0\subsetneq A_1\subsetneq A_{2}\subsetneq \cdots \subsetneq A_{N} \subset V$$
with
$$\dim(A_i) = a_i.$$
Define
$$ F: = V^*\otimes E= \Hom(V\otimes \mathcal O_X, E)$$
Letting $\rank(F) = f+1$ we then have  
$$ f  = (e+N)e-1\ge 2b+3.$$  
Inside $F$ define
\begin{equation}\label{eq:defhatC}
 \hat{C}: = \{ \sigma\in F_x : \dim( \operatorname{ker}(\sigma(x)) \cap A_i ) \ge i \text{ for 
all $i=1,\ldots, N$ and $x\in X$}\}
\end{equation}
which is a cone in $F$.  Now set
$$ P: = \mathbb P_{\sub}(F) \ {\text 
{and} } \  C = [\hat{C}] \subset P.$$

\begin{proposition}\label{prop:conebasics}
\begin{enumerate}[(a)] 
\item $C$ has codimension $b$ and dimension $n :=\dim C =  f+d-b$.  
\item $C$ is locally a product over $X$.
\item $C$ has irreducible fibers over $X$.  
\item We have
$$\pi_* c_{f}(U|_C) = s_{\lambda} (E),$$
where $U$ denotes the universal quotient bundle on $P$ as in Section \ref{sec:cone_classes}. 
\end{enumerate}
\end{proposition}
\begin{proof}
All of this is standard (e.g. \cite[(8.12)]{Lazbook2} which is written for the case $|\lambda|=d$ but that makes no essential difference).  For completeness we show precisely where this is contained in \cite{FultonIT} (much of which is merely a translation of notation).

Let $\pi:F=V^*\otimes E\to X$ be the projection and consider the tautological section
$$ u: \pi^*V \to \pi^* E.$$
Then
\begin{align*}\hat{C}: &= \{ \sigma\in F_x : \dim( \operatorname{ker}(\sigma(x) \cap A_i ) \ge i \text{ for 
all $i$ }\}\\
&= \{ \sigma\in F : \dim \operatorname{ker}(u(\sigma)) \cap \pi^*A_i \ge i \text { for all $i$ }\}.
\end{align*}
So, in the notation of \cite[p243 and Remark 14.3]{FultonIT} our $\hat{C}$ is written as
$$\hat{C} = \Omega(\pi^* A; u).$$

Now, since $E$ and $F$ are locally trivial, one sees that $\hat{C}$ is locally a product, with fibre $Z$ given by the case that $X$ is a single point.  This is the ``universal case" discussed in  \cite[p250, final paragraph]{FultonIT}   and in \cite[Lemma A.7.2]{FultonIT} is the precise statement that implies $Z$ is irreducible and of codimension
 $$\sum_{i=1}^N e-a_i +i =\sum_{i=1}^N \lambda_i = |\lambda| = b.$$


Next let $t:X\to F$ be the zero section, which we think of as a regular embedding of $X$ of codimension $f+1$.  Then $\sigma: = t^*u$ is the zero section $\sigma:V\to E$.  Then using \cite[Remark 14.3]{FultonIT} 
$$\boldsymbol{\Omega}(A;\sigma) = t^{!} [ \Omega(p^*A;u)] = t^{!}(\hat{C})$$
and then using \cite[Theorem 14.3(a)]{FultonIT} gives
$$ t^{!}(\hat{C}) =  \boldsymbol{\Omega}(A;\sigma) = s_{\lambda}(E)$$
where $$t^{!}: A_*(F)\to A_{*-f-1}(X)$$
is the Gysin morphism, as defined in \cite[Section 6.2]{FultonIT}.  Since we have changed notation from that in \cite{FultonIT} we include the following table as a guide.

\begin{table}[h]
\begin{tabular}{|L|L|}
\hline
  \text{\cite[Rmk.\ 14.3]{FultonIT}} & \text{This paper} \\
 \hline
 d & N\\
 n & d\\
 \lambda_i & \lambda_i\\  
 h &  b\\ 
 a_i & a_i\\
 A_i & A_i \\
 \underline{A} & A\\
 F & E\\
 f & e\\
 E &  V\\
 e &  e+N\\
 H & F \\
 p & \pi\\
 \sigma & \sigma\\
 t_{\sigma} & t\\
 \hline
 \end{tabular}
\end{table}


The point finally is that since $t$ is the zero section we can express $t^{!}(\hat{C})$ as the pushforward of the top Chern class of the tautological bundle on $\mathbb P(F)$ restricted to $\hat{C}$.  To see this let
$$\pi: P'= \mathbb P_{\sub}(F\oplus \mathbb C)\to X$$
be the projective completion of $F$, with universal quotient bundle $U'$ which has rank $f+1$.     Let $C'$ be the closure of $\hat{C}$ inside $P'$.  Then $C'$ has the property that the restriction of $C'$ to $F\subset P'$ is equal to $\hat{C}$.  So \cite[Proposition 3.3]{FultonIT} gives
$$t^! \hat{C} = \pi_*(c_{f+1}(U')|_{C'})$$
(we observe that the cited work states this formula for $t^*\hat{C}$, but that is equal to the Gysin morphism $t^{!}$ in this case, see \cite[Remark 6.2.1]{FultonIT}).

Thus in total we have
$$ s_{\lambda}(E) = \pi_*(c_{f+1}(U')|_{C'})$$

Now clearly each fiber of $\hat{C}$ is not contained in the zero section (for dimension reasons alone).   So by \cite[Proof of Corollary 8.1.14]{Lazbook2}, if $U$ denotes the tautological bundle on $\mathbb P(F)$ then
$$ \pi_*(c_{f+1}(U')|_{C'} = \pi_* (c_{f}(U)|_{C})$$
and the proof of (d) is complete since   $n=f+d-b$. 
\end{proof}


\subsection{An extension}
We now extend this to both the derived Schur classes $s_{\lambda}^{(i)}$ from Definition \ref{def:lowerschur} and also to the case of $\mathbb R$-twisted bundles.    As in the previous section suppose   $d\ge 4$, 
$e  \ge 2$, 
$0\le \lambda_N\le \lambda_{N-1}\le \cdots \le \lambda_1\le e$, $1\le b\le N$ and $b\le d$.

Let $E' =  E\langle \delta \rangle$ be an $\mathbb R$-twisted bundle, where $E$ is a vector bundle and $\delta\in H^{1,1}(X;\mathbb R)$.  Recall we identify $P':=\mathbb P_{\sub}(E')$ with $P= \mathbb P_{\sub}(E)$
and if $U$ is the universal quotient bundle on $P$ then the universal quotient bundle on $P'$ is defined to be 
$$U': = U\langle \pi^* \delta \rangle.$$
 Now consider the same cone  
    $$[C]\subset P' = P$$
    as in \eqref{eq:defhatC}. As before $n:=\dim C=f+d-b$ and $ f  = (e+N)e-1$.

\begin{proposition}
\label{prop:pushforwardlower} 
Under the above notation, for $0\le i\le b$  it holds that
\begin{equation}\pi_* c_{f-i}(U'|_C)= s_{\lambda}^{(i)}(E') \label{eq:pushforwardlower}\end{equation}\end{proposition}
\begin{proof}
We prove this first in the case $\delta=0$, so $E'=E$ and $U'=U$. 
We note first that the construction of the cone $C=C(X, E,\lambda,V,(A_i)_i)$ over $X$ in section \ref{subsec:Schur} depending on  $E$, $\lambda$, $V$ and $ (A_i)_i$ commutes with base change, that is if $\phi:\tilde{X}\to X$ is a morphism between projective manifolds then the cone $\tilde{C}=C(\tilde X,\phi^*(E),\lambda,V,(A_i)_i)$ over $\tilde{X}$ sits in a cartesian square
 $$
  \xymatrix{ \tilde C \ar[r]^{\psi}\ar[d]^{\tilde{\pi}} 
& C \ar[d] ^{\pi}\\
\tilde{X}\ar[r]^{\phi} & X.
 }
  $$ 
  If we can prove the desired formula \eqref{eq:pushforwardlower} for $\phi^*(E)$ over $\tilde{X}$ and if $\phi$ is flat and such that $\phi^*:H^*(X,\Q)\to H^*(\tilde{X},\Q)$ is injective, then the formula will be also valid for $E$ over $X$. Indeed, we would have
  $$ \phi^*(\pi_* c_{f-i}(U|_C)-s_{\lambda}^{(i)}(E))= \tilde{\pi}_*\psi^*(c_{f-i}(U|_C))-s_{\lambda}^{(i)}(\phi^*E)=$$ 
  $$\tilde{\pi}_*(c_{f-i}((\psi^*U)|_{\tilde C}))-s_{\lambda}^{(i)}(\phi^*E)= s_{\lambda}^{(i)}(\phi^*E)-s_{\lambda}^{(i)}(\phi^*E)=0
  $$
We thus may reduce ourselves, as we will, to the situation when $d\ge 2b$, by taking $\tilde{X} = X\times \P^b$, for instance.  
 
 Take now an ample integer class $\epsilon\in H^{1,1}(X;\mathbb Q)$.  Then using the projection formula
\begin{equation}\pi_* c_{f}(U\langle \pi^* \epsilon \rangle|_C\rangle) = \pi_*\left( \sum_{i=0}^{f} c_{f-i}(U) (\pi^*\epsilon)^i\right) =  \sum_{i=0}^{b} (\pi_* c_{f-i}(U)).\epsilon^i.\label{eq:proppushfowardlower1}\end{equation}
(Note that for dimension reasons $\pi_* c_{f-i}(U|_C)=0$ if $i>b$.)
On the other hand replacing $E$ by $E\otimes \mathcal O(\epsilon)$ does not change $P,C$ but has the effect of replacing $U$ by $U\otimes \mathcal O(\epsilon)$.   Since $U\otimes \mathcal O(\epsilon)$ is a genuine bundle (not $\mathbb R$-twisted),   Proposition \ref{prop:conebasics}(d) applies to give
\begin{equation}\pi_* c_{f}(U\langle \pi^*\epsilon\rangle)  = \pi_* c_{f}(U\otimes \pi^*\mathcal O(\epsilon)) =  s_{\lambda}(E\otimes \mathcal O(\epsilon)) = \sum_{i=0}^{b} s_{\lambda}^{(i)}(E). \epsilon^i\label{eq:proppushfowardlower2}\end{equation}
where the last equation uses the definition of $s_{\lambda}^{(i)}$.  Comparing  \eqref{eq:proppushfowardlower1} and \eqref{eq:proppushfowardlower2} yields 
$$\sum_{i=0}^{b} (\pi_* c_{f-i}(U)-s_{\lambda}^{(i)}(E)).\epsilon^i=0$$ and replacing $\epsilon$ by $t\epsilon$ also
$$\sum_{i=0}^{b} t^i(\pi_* c_{f-i}(U)-s_{\lambda}^{(i)}(E)).\epsilon^i=0$$
for any $t\in\{0,...,b\}$. (Here we set $0^0$  to be $1$.) The Vandermonde matrix $(t^i)_{0\le i,t\le b}$ being invertible,
 we find   
 $$(\pi_* c_{f-i}(U)).\epsilon^i=(s_{\lambda}^{(i)}(E)).\epsilon^i $$
 for all $i\in\{0,...,b\}$ and formula \eqref{eq:pushforwardlower} follows now by Hard Lefschetz, since $b$ was supposed not to exceed $d/2$.

The result for general $\delta$ is now given by the following formal computation:
\begin{align*}
\pi_* c_{f-i}(U'|_C) &= \pi_* c_{f-i}(U\langle \pi^* \delta \rangle|_C)  \\
 \tag{by \eqref{eq:chernclasstensorproduct:full}}\\
&= \sum_{j=0}^{f-i} \binom{i+j}{j} \pi_* c_{f-i-j}(U) \delta^{j}  \tag{by \eqref{eq:chernclasstensorproduct:full}}\\
& = \sum_{j=0}^{b-i} \binom{i+j}{j} s_{\lambda}^{(i+j)}(E) \delta^{j}  
\\
& = \sum_{k=i}^{b} \binom{k}{i} s_{\lambda}^{(k)}(E) \delta^{k-i}  \\
& = s_{\lambda}^{(i)}(E\langle \delta \rangle) \tag{by \eqref{eq:lowercherntwistexpansion}}
\end{align*}
\end{proof}

\subsection{Proof of Hodge-Riemann Property for Schur classes}

\begin{theorem} \label{thm:schurhodgeriemann}
Let $X$ be a projective manifold of dimension $d\ge 2$ and $E$ an ample $\mathbb R$-twisted vector bundle of rank $e$.  Let $\lambda$ be a partition of $d-2$ with $0\le \lambda_N\le \lambda_{N-1}\le \cdots \le \lambda_1\le e$.    Then for any ample class $h$ and any $0\le i\le d-2$,
$$s_{\lambda}^{(i)}(E).h^{i}\in H^{d-2,d-2}(X;\mathbb R)$$
has the Hodge-Riemann property with respect to $h$.  

In particular, applying this when $i=0$, the Schur class $s_{\lambda}(E)$ has the Hodge-Riemann property.
\end{theorem}
\begin{proof}
When $d=2$ the statement follows from the classical Hodge-Riemann bilinear relations, and also when $d=3$ for then the only non-zero Schur class is $c_1(E)$ which is ample.  Thus we may assume $d\ge 4$, and there is no loss in generality in assuming $h$ is very ample.  Furthermore, the statement clearly holds for $e=1$ so we may suppose that $e\ge 2$.

Since $E$ is ample so is $F:= V^*\otimes E$.  Moreover Proposition \ref{prop:conebasics}(b,c,d) tell us that $C\subset \mathbb P_{\sub}(F)=:P$ is irreducible, locally a product and of dimension
$$ n: = \dim C = f+2  \ge d+2$$

Now using Proposition \ref{prop:pushforwardlower} and the projection formula, for all $\alpha,\alpha'\in H^{1,1}(X;\mathbb R)$,
$$\int_X \alpha s_{\lambda}^{i}(E) h^{i} \alpha' = \int_{C} (\pi^*\alpha)c_{n-2-i}(U) (\pi^*h)^{i} (\pi^*\alpha')$$
where $U$ is the universal quotient $\mathbb R$-twisted bundle on $P$.   But our assumption that $E$ is ample implies $F$ is also ample, and thus the result we want follows from Theorem \ref{thm:maincone}.
\end{proof}

\begin{remark}
The Hodge-Riemann property also holds for Schur classes of filtered bundles as considered in \cite{Fultonfiltered}.  In fact in \cite[p630]{Fultonfiltered} it is shown how these classes can be written as cone classes just as in \eqref{prop:conebasics}, so Theorem \ref{thm:maincone} applies in this setting as well.
\end{remark}

\section{An application to cones of cycles}\label{sec:conesofcycles}
The following application was suggested by Brian  Lehmann, and answers in part questions posed in \cite[Problem 6.6]{DELV11} and \cite[Sec 6.2]{Lehmann} concerning cones of cycles of arbitrary codimension.

 On a projective manifold $X$ of dimension $d$ define the \emph{the cone of nef classes of codimension $k$}
$$\Nef^k(X)\subset H^{k,k}(X;\mathbb R)$$
as the cone spanned by those classes $\alpha$ such that $\int_Z \alpha\ge 0$ for all subvarieties $Z\subset X$ of dimension $k$.    One can also define a cone 
$$\operatorname{Schur}^k_{\Nef}\subset N^k(X)$$
as the closed convex cone generated by all Schur classes $s_{\lambda}(E)$ where $E$ is a nef vector bundle on $X$ and $\lambda$ is a partition of $k$.

So, from the work of Fulton-Lazarsfeld \cite{FultonLazarsfeld} we certainly have
$$\operatorname{Schur}^k_{\Nef}(X)\subset \Nef^k(X),$$
and in this section we will show that this inclusion may be strict.


To do so we build on the analysis in \cite{DELV11} which contains a complete description of $\Nef^2(A\times A)$ where $(A,\theta)$ is a very general principally polarized  abelian surface.  
Using their notation, 
$N^1(A\times A)_{\mathbb R}$ has rank 3 with basis $\theta_1,\theta_2,\lambda:=c_1(\mathcal P)$ where $\theta_1,\theta_2$ are the pull-backs of $\theta$ from the two factors of $A\times A$ and $\mathcal P$ is the Poincare bundle on $A$ \cite[Prop. 3.1]{DELV11}.   Moreover \cite[Section 4]{DELV11} we know that $N^2(A\times A)_{\mathbb R}$ has rank 6, with basis $\theta_1^2,\theta_1\theta_2, \theta_2^2,\theta_1\lambda, \theta_2\lambda,\lambda^2$ and the only non-zero products of degree $4$ of these classes are
\begin{equation} \theta_1^2\theta_2^2 = 4 \quad \theta_1\theta_2\lambda^2 = -4 \quad \lambda^4 = 24.\label{eq:degree4products}
\end{equation}
We define
$$\mu:= 8\theta_1\theta_2 + 3\lambda^2.$$

\begin{lemma}\label{lem:basicsaboutmu}\
\begin{enumerate}
\item $\mu$ spans a one-dimensional face of the boundary of $\Nef^2(A\times A)$.
\item The intersection form defined by $\mu$ has the Hard-Lefschetz property but not the Hodge-Riemann property.
\end{enumerate}
\end{lemma}
\begin{proof}
The first statement follows from the explicit description of $\Nef^2(A\times A)$ given in \cite[Prop 4.2]{DELV11}.   They show that a class
\begin{equation} a_1 \theta_1^2 + a_2 \theta_1\theta_2 + a_3\theta_2^2 + a_4 \theta_1\lambda + a_5 \theta_2\lambda + a_6\lambda^2\label{eq:generalclassNef2}\end{equation}
is in $\Nef^2(X)$ if and only if
\begin{align}
a_1,a_3 & \ge 0, \label{eq:nefboundary1}\\
a_2&\ge a_6, \label{eq:nefboundary2}\\
4a_1(a_2-a_6) &\ge a_4^2, \label{eq:nefboundary4}\\
4a_3(a_2-a_6)&\ge a_5^2, \text{ and }\label{eq:nefboundary5}
\end{align}
$$(a_5b^2 + (a_2-6a_6)b +a_4)^2 \le 4(a_3b^2  -a_5b + a_2 - a_6)((a_2-a_6)b^2 - a_4b + a_1)$$
for all $b\in \mathbb R$.
Note that when $a_1=a_3=a_4=a_5=0$ these inequalities reduce to
$$-\frac{1}{4} a_2 \le a_6\le \frac{3}{8}a_2.$$
From this it is clear that $\mu = 8\theta_1\theta_2 + 3\lambda^2\in \Nef^2(A\times A)$.  On the other hand, if $\mu = \sum_i t^{(i)} v^{(i)}$ is a convex combination of nef classes written as
$$v^{(i)} =a_1^{(i)} \theta_1^2 + a_2^{(i)} \theta_1\theta_2 + a_3^{(i)}\theta_2^2 + a_4^{(i)} \theta_1\lambda + a_5^{(i)} \theta_2\lambda + a_6^{(i)}\lambda^2$$
then \eqref{eq:nefboundary1} implies that $a_1^{(i)} = a_3^{(i)}=0$ for all $i$, and then (\ref{eq:nefboundary4}, \ref{eq:nefboundary5})
imply $a_4^{(i)} = a_5^{(i)}=0$ for all $i$.   Thus we in fact have $-\frac{1}{4} a_2^{(i)} \le a_6^{(i)} \le \frac{3}{8} a_2^{(i)}$ for all $i$, and since $\mu$ lies on one extremity of this inequality we must have $a_6^{(i)} = \frac{3}{8}a_2^{(i)}$ for all $i$.  Hence each $v_i$ is a scalar multiple of $\mu$ proving (1).

For (2) we observe that  \eqref{eq:degree4products} implies the intersection pairing of $\mu$ on  $N^1(A\times A)_{\mathbb R}$ taken with respect to the basis $\theta_1,\theta_2,\lambda$ has matrix
$$ Q: = 20\left(\begin{array}{ccc} 0 & 1 & 0 \\ 1 & 0 & 0 \\ 0&0& 2\end{array}\right)$$
which has strictly negative determinant.  Thus $\mu$ has the Hard-Lefschetz property, but cannot have the Hodge-Riemann property (which would require $Q$ to have signature $(1,2)$ and thus strictly positive determinant).
\end{proof}

\begin{proposition}
If $A$ is a very general  principally polarized  abelian surface then
$$\operatorname{Schur}^2_{\Nef}(A\times A)\subsetneq \Nef^2(A\times A).$$
\end{proposition}

\begin{proof}
Consider an affine hyperplane 
$H $ in $ N^2(A\times A)$ such that $H\cap\R\mu=\{\mu\}$ and the closed convex sets 
\begin{align*}
C_1 &:= \Nef^2(A\times A)\cap H \\
C_2 & :=\operatorname{Schur}^2_{\Nef}(A\times A)\cap H.\end{align*}
The goal is to show that $\mu$ lies in $C_1$ but not in  $C_2$ and thus $\operatorname{Schur}^2_{\Nef}(A\times A)\subsetneq \Nef^2(A\times A).$
Let $\Lambda\subset H$ be the closure of the  subset of $C_1$ consisting of those classes in $C_1$ having the Hodge-Riemann property. By our main result $\Lambda$ contains all positive scalar multiples in $C_1$ of classes of the form $s_{\lambda}(E)$ for nef vector bundles   $E$ on $X$ and $|\lambda|=2$.  In particular $C_2\subset \overline{\operatorname{Conv}}(\Lambda)$. 

By Lemma \ref{lem:basicsaboutmu} $\mu$ is an extreme point of $C_1$ not lying in $\Lambda$. The result we want is now an elementary statement about convex sets in finite dimensional vector spaces which we give in Lemma \ref{lem:sillycones}.
\end{proof} 

\begin{lemma}\label{lem:sillycones}
Let $\Lambda$ be a non-empty closed set in a finite dimensional vector space $H$ and let   $C:=\overline{\operatorname{Conv}(\Lambda)}$ be its closed convex hull. Then all extreme points of $C$ belong to $\Lambda$.
\end{lemma}
\begin{proof}
By a result of Straszewicz \cite{Straszewicz} every extreme point of a  closed convex set  $C$ in a finite dimensional vector space $H$ is a limit of exposed points of $C$. So it will be enough to show that the exposed points of $C$ belong to $\Lambda$, since $\Lambda$ is closed. Recall that a point $x\in C$ is called {\em exposed} if there exists an affine function $f$ on $H$ such that $C\cap\{ f=0\}=\{ x\}$, or in other words if there exists a supporting hyperplane $H_0$ for $C$ with  $C\cap H_0=\{ x\}$. 

So let $x$ be an exposed point of $C$ with supporting affine function $f$ and supporting hyperplane $H_0=\{ f=0\}$ and such that $C\subset\{ f\ge 0\}$. We fix a scalar product on $H$. We consider the sets   $V_t:=C\cap\{0\le f< t\}$ for $t>0$ and will show that they form a neighbourhood basis of $x$ in $C$. Their complements $C\setminus V_t$ in $C$ cannot contain $\Lambda$ since they are closed and convex and do not contain $x$. From this it follows that $x$ is in $\Lambda$.

It remains to show that the system $(V_t)_{t>0}$ is a neighbourhood basis for $x$ in $C$. Take any compact hypercube $W$ in $H$ with one (top dimensional) face $F$ on $H_0$ such that $F$ is centred at $x$ and such that $f$ is non-negative on $W$. Then $W$ is a neighbourhood of $x$ in $C$. Its face $F$ meets $C$ only in $x$. The boundary of $F$ is compact and disjoint from $C$ and hence has a positive distance $d$ to $C$.  If we take $t\le d$, then $V_t$ is completely contained in $W$. Since we may choose $W$ arbitrarily small our claim follows.
\end{proof} 
%
%
%

\section{Higher rank Khovanskii-Teissier inequalities}\label{sec:logconcavity}

\begin{lemma}\label{lem:FLforlowerschur}
Let $E$ be an ample vector bundle on $X$ of rank $e\ge d$ where $d = \dim X\ge 2$ and  let $\mu$ be a partition of $e$.  Then
$$\int_X s_{\mu}^{(e-d)}(E)>0.$$
\end{lemma}
\begin{proof}
Write $e=d+k$ and let $\sigma$ denote the class of the hyperplane class on $\mathbb P^k$.  The bundle $E':= E\boxtimes \mathcal O(\sigma)$ on $X\times \mathbb P^k$ is ample so by Fulton-Lazarsfeld $\int_{X\times \mathbb P^k} s_{\mu}(E')>0$.   Now 
$$s_{\mu}(E') = s_{\mu}^{(e-d)}(E) \sigma^k$$
(we have used here that $s_{\mu}^{(i)}(E)=0$ if $e-i=|\mu|-i>d$ and also that $\sigma^j=0$ if $j>k$).   The result follows.
\end{proof}

\begin{proposition}\label{prop:logconcave}
Let $E$ be an ample vector bundle of rank $e\ge d$ where $d = \dim X\ge 2$ and $\alpha\in H^{1,1}(X;\mathbb R)$.  Let $\mu$ be a partition of $e$.   Then
\begin{equation}\int_X s_{\mu}^{(e-d)}(E) \int_X s_{\mu}^{(e-d+2)}(E) \alpha^2 \le  \left( \int_X s_{\mu}^{(e-d+1)}(E) \alpha \right)^2\label{eq:logconcaveI}\end{equation}
with equality if and only if $\alpha=0$.
\end{proposition}
\begin{proof}
Write $e= d+k$ so $k\ge 0$,  and set $X' = X\times \mathbb P^{k+2}$.  
Denote by $\tau$ be the hyperplane class on $\mathbb P^{k+2}$ and set   $$E' := E\boxtimes \mathcal O(\tau)$$  which is ample.  Clearly $\rank(E')=e =d+k  = \dim X'-2$.  Moreover by definition of the derived Schur-classes,
\begin{equation}s_{\mu}(E') = s_{\mu}^{(k)}(E)\tau^k + s_{\mu}^{(k+1)}(E) \tau^{k+1}+ s_{\mu}^{(k+2)}(E)\tau^{k+2}.\label{eq:cnonXPk2}
\end{equation}
(we have used here that $s_{\mu}^{(i)}=0$ if $e-i=|\mu|-i>d$ and $\tau^j=0$ if $j>k+2$).  In particular
$$\int_{X'} s_{\mu}(E') \tau^2  = \int_X s_{\mu}^{(k)}(E)> 0$$
where the last inequality follows from Lemma \ref{lem:FLforlowerschur}.  So we may apply the Hodge-Index inequality  (cf. 
Definition-Lemma \ref{deflemma:HRofquadraticform}) 
for $s_{\mu}(E')$ which gives
$$\int_{X'} \beta^2  s_{\mu}(E') \int_{X'} \tau^2 s_\mu(E') \le \left(\int_{X'} \beta \tau s_{\mu}(E')\right)^2 \text{ for all } \beta\in H^{1,1}(X';\mathbb R)$$
with equality if and only if $\beta$ is proportional to $\tau$.    In particular this applies when $\beta =\alpha\in H^{1,1}(X)$, and from \eqref{eq:cnonXPk2}
$$\int_{X'} \alpha^2 s_{\mu}(E') = \int_X \alpha^2  s_{\mu}^{(k+2)}(E)$$
$$\int_{X'} \alpha\tau s_{\mu}(E') = \int_X \alpha s_{\mu}^{(k+1)}(E).$$
Putting this altogether yields \eqref{eq:logconcaveI}.  Moreover equality holds in \eqref{eq:logconcaveI} if and only if $\alpha$ is proportional to $\tau$, which happens if and only if $\alpha=0$.


\end{proof}

\begin{remark}
Consider the case $\dim X=2$ and $E$ is ample of rank at least $2$ and $\mu_1=2$.    Then \eqref{eq:logconcaveI} becomes 
$$\int_X c_2(E) \int_X \alpha^2 \le  \left(\int_X c_1(E) \alpha \right)^2$$
with equality if and only if $\alpha=0$.  In particular this holds when $\alpha=c_1(E)$, in which case this inequality simplifies to
$$\int_X c_1(E)^2 -c_2(E) > 0.$$
This is as expected from \cite{FultonLazarsfeld} since $c_1(E)^2-c_2(E)$ is a Schur class.
\end{remark}

\begin{theorem}[Log-concavity for Schur numbers]\label{thm:logconcaveschur}
Let $X$ be projective of dimension $d\ge 2$, let  $h\in H^{1,1}(X,\mathbb Z)$ be an integral ample class and let $E$ be an ample vector bundle on $X$ of rank  $e\ge d$ and let $\mu$ be a partition of $e$.   Then the function
$$ i \mapsto \int_X s_{\mu}^{(e-i)}(E) h^{d-i} \text{ for }i=0,\ldots, d$$
is strictly log-concave.  
\end{theorem}

Note that in the particular case of Example \ref{ex: Chern classes} we obtain the stated Theorem \ref{thm:intrologconvex}.

\begin{proof}
Without loss of generality we may assume $h$ is very ample.  Then for each $i=2,\ldots,d$ the class $h^{d-i}$ is represented by a smooth submanifold $Y\subset X$ of dimension $i$.   Applying \eqref{eq:logconcaveI} to $E|_Y$ (with $\alpha$ replaced by $h|_Y$) gives
\begin{equation}\label{eq:logconcaverestricted}\int_Y s^{(e-i)}(E|_Y)  \int_Y s^{(e-i+2)}(E)h^2|_Y < \left(\int_Y s^{(e-i+1)}(E) h|_Y \right)^2.\end{equation}
where we have also used functoriality of the derived Schur classes.  Said another way,
\begin{equation}\label{eq:logconcaverestricted:repeat}\int_X s^{(e-i)}(E) h^{d-i}  \int_X s^{(e-i+2)}(E) h^{d-i+2} < \left(\int_X s^{(e-j+1)}(E) h^{d-i+1}\right)^2.\end{equation}
Thus defining 
$$f(i): =\log \int_X s^{(e-i)}(E) h^{d-i},$$
and taking the logarithm of \eqref{eq:logconcaverestricted:repeat} yields 
$$ \frac{1}{2}(f(i) + f(i-2)) < f(i-1) \text{ for } i=2,\ldots, d.$$
The conclusion we want about $f$ is then a formal statement about functions with this property (Lemma \ref{lem:concavityformal}).
\end{proof}

\begin{lemma}\label{lem:concavityformal}
Let $f:\{0,\ldots,d\} \to \mathbb R$ be a function such that
$$ \frac{1}{2}(f(i) + f(i-2)) < f(i-1)  \text{ for } i=2,\ldots d.$$
Then for any $0\le i<j<k\le d$ if $t$ is defined so $j=ti+(1-t)k$ 
\begin{equation}
tf(i) + (1-t)f(k) < f(j).\label{eq:fconvex}\end{equation}
\end{lemma}
The conclusion of this Lemma just says that the closed polygonal chain obtained by connecting successive points of the graph of $f$ to which one adds the base segment $$[(0,f(0)),(d,f(d))]$$ is a (strictly) convex polygon in $\R^2$ lying ``above" the base segment. Its proof is elementary and left to the reader.

%

\begin{remark}\label{rem:Khovanskii}
The previous theorem generalises the Khovanskii-Teissier inequalities \cite{Teissier} which state the following:  let $\alpha,\beta\in H^{1,1}(X,\mathbb Z)$ be nef classes on a projective manifold $X$ of dimension $d$ and set 
$$ s_i: = \int_X \alpha^i \beta^{d-i} \text{ for } i=0,\ldots, d.$$
Then the function $i\mapsto s_i$ is log-concave.  To see how this follows from Theorem \ref{thm:logconcaveschur}, notice first that by continuity we may as well asssume that $\alpha,\beta$ are ample, and replacing $\alpha$ with a positive multiple if necessary (which does not change the statement) we may assume that $\mathcal O(\alpha)$ is very ample.  Thus there is a surjection $\mathcal O^{\oplus e+1} \to \mathcal O(\alpha)$ for some $e\ge d$, and dualizing gives a short exact sequence
$$0\to \mathcal O(-\alpha) \to \mathcal O^{\oplus {e+1}} \to E\to 0.$$
Then $E$ is nef, which is a limit of ample $\mathbb R$-bundles, and thus Theorem \ref{thm:logconcaveschur} implies the map $i\mapsto \int_X c_i(E) h^{d-i}$ is log-concave (but not necessarily strictly).  Finally since $c_i(E) =\alpha^i$ so $\int_X c_i(E) \beta^{d-i}  = s_i$ we have the Khovanskii-Teissier inequalities.
\end{remark}

\section{Schur polynomials of K\"ahler forms}\label{sec:Dihn}

Suppose $E$ splits as a sum of line bundle $E=\oplus_i L_i$ and set $a_i:= c_1(L_i)$.  Then $c(E) = \Pi_{i} (1+a_i)$ and the Schur classes $s_{\lambda}(E)$ are universal symmetric polynomials in the elementary classes $\{a_i\}$, which we write as $s_{\lambda}(a_1,\ldots,a_e)$.     Even without the vector bundle one can ask for the Hodge-Riemann property when the $a_i$ are replaced with K\"ahler classes:
 
 \begin{question}\label{question:Kahlerforms}
 Suppose that $\hat\omega_1,\ldots,\hat\omega_e$ are K\"ahler  classes 
 on a compact complex manifold $X$ of dimension $d$, and $(\lambda,e,d)$ are in the same range as required by Theorem \ref{thm:schurhodgeriemann}.   Does the class
 $$ s_{\lambda}(\hat\omega_1,\ldots,\hat\omega_e) \in H^{d-2,d-2}(X;\mathbb R)$$
 have the Hodge-Riemann property?
 \end{question}
 
By the main result of  Dinh-Nguy\^en \cite{Dinh2} one may relate this to  the following similar question in linear algebra.  Let $V$ be a $d$-dimensional complex vector space, let $V_\R$ its underlying real vector space and let $U$ be a lattice in $V$. 
Following \cite[Sections 1-2]{MumfordAbelianVarieties} we denote by $T:=\Hom_\C(V,\C)$, $\bar T:=\Hom_{\C-\operatorname{antilin}}(V,\C)$ and $T^{p,q}=\bigwedge^pT\otimes\bigwedge^q\bar T$ the spaces of $(1,0)$, $(0,1)$ and $(p,q)$-forms on $V$, respectively. 
Elements in $T^{p,p}$ may be viewed as sesquilinear forms on $\bigwedge^pV$. Such an element is said to be {\em real} if the corresponding form is Hermitian, and $T^{p,p}_\R$ denotes the space of real $(p,p)$-forms.
We say that an element $\omega$ in $T^{1,1}_\R$ is a {\em K\"ahler form} if for some choice of a basis for $V$ we can write
$$\omega=i\sum_{j=1}^dd z_j\wedge d \bar z_j.$$
We will denote by $K(V)$ the cone of K\"ahler forms on $V$. If a K\"ahler form $\omega$ has been fixed we will call the pair $(V, \omega)$ a {\em polarized vector space}.
Recall that in each $T^{p,p}_\R$ one has positive cones generated by forms of the type $i^{p^2}\alpha\wedge\bar\alpha$, for $\alpha\in T^{p,0}$. 
A positive $(p,p)$-form is said to be {\em strictly positive} if its restriction to any $p$-dimensional complex subspace of $V$ is non-zero. 
Any non-zero positive $(d,d)$-form $\eta$ is strictly positive and defines an  isomorphism $\int:T^{d,d}_\R\to\R$ which  preserves positivity. We will always assume this when using this notation. 
We say that an element $\omega$ in $T^{1,1}_\R$ is {\em integral}, respectively {\em rational}, if its imaginary part, which is an alternating skew-symmetric form on $V_\R$, takes values in $\Z$, respectively in $\Q$, on $U\times U$. 
 Finally for a polarized vector space $(V,\omega)$ an element $\Omega\in T^{d-2,d-2}_{\mathbb R}$ is said to have the \emph{Hodge-Riemann property} if $\int\Omega\wedge\omega^2>0$ and if the blinear form
$$(\alpha,\alpha')\mapsto \int\alpha \wedge \Omega\wedge  \alpha'$$
has signature $(1,d-1)$. 

We can now formulate the linear algebraic analogue of Question \ref{question:Kahlerforms}. 

\begin{question}\label{question:LinearAlgebra}
 Suppose that $\omega_1,\ldots,\omega_e$ are K\"ahler forms on a complex vector space $V$ of dimension $d$, and $(\lambda,e,d)$ are in the same range as required by Theorem \ref{thm:schurhodgeriemann}.  Does
 $$ s_{\lambda}(\omega_1,\ldots,\omega_e) \in T^{d-2,d-2}$$
 have the Hodge-Riemann property?
 \end{question}

If $X$ is the torus $V/U$ then using the natural isomorphisms $H^q(X, \Omega^p)\cong T^{p,q}$ one immediately sees that Question \ref{question:Kahlerforms} for the manifold $X$ is equivalent to Question \ref{question:LinearAlgebra} for the vector space $V$.  Since Chern classes of ample line bundles on $X$ are integer K\"ahler classes, we may use this observation in combination to Theorem \ref{thm:schurhodgeriemann} to get:

\begin{corollary} \label{cor:linearalgebra1}
Let $\omega_1,\ldots,\omega_e$ be rational K\"ahler forms on the $d$-dimensional complex vector space $V$ and  let $(\lambda,e,d)$ be in the same range as required by Theorem \ref{thm:schurhodgeriemann}.   Then the form 
$s_{\lambda}(\omega_1,\ldots,\omega_e)$ has the Hodge-Riemann property. In particular the linear map
$$T^{1,1}\to T^{d-1,d-1}, \ \eta\mapsto \eta\wedge s_{\lambda}(\omega_1,\ldots,\omega_e),$$ 
is invertible.
\end{corollary}

The theorem of  Dinh and Nguy\^en goes in the opposite direction. For bi-degree $(d-2,d-2)$ it says that the cohomology class of a closed smooth positive $(d-2,d-2)$-form $\Omega$ on compact K\"ahler manifold $X$ has the Hodge-Riemann property if for all $x\in X$ the form $\Omega(x)$ is in the \emph{Hodge-Riemann cone} of $T_{(X,x)}$, \cite[Theorem 1.1]{Dinh}. 
They define the Hodge-Riemann cone $HR^{d-2,d-2}\subset T^{d-2,d-2}_\R$ for a polarized vector space $(V,\omega)$  of dimension $d$ by saying that   a $(d-2,d-2)$-form $\Omega$ lies in $HR^{d-2,d-2}$ if there exists a continuous deformation $\Omega_t\in T^{d-2,d-2}_R$, $t\in[0,1]$, such that $\Omega_0=\Omega$, $\Omega_1=\omega^{d-2}$, $\Omega_t\wedge\omega^2\neq0$ for all $t\in[0,1]$ and the map
\begin{equation}
T^{1,1}\to T^{d-1,d-1}, \ \eta\mapsto \eta\wedge\Omega_t
\label{eq:HRCone1}\end{equation}
is an isomorphism for all  $t\in[0,1]$.

Thus we see that an affirmative answer to Question \ref{question:LinearAlgebra} for a triple $(\lambda, e,d)$ implies an affirmative answer to Question \ref{question:Kahlerforms} for the same triple. 

We now answer  Question \ref{question:Kahlerforms} affirmatively in the special case when $e=2$ and $s_{\lambda} = s_{(1,1,\ldots,1)}$  and hope to consider  the general case in the future. We note that in degree $k$ the class $ s_{(1,1,\ldots,1)}(E)$ for a vector bundle is the $k$-th Segre class of its  dual, $s_k(E^*)$, \cite[Example 8.3.5]{Lazbook2}.

%

\begin{proposition}\label{prop:Segre}
Let $X$ be a compact K\"ahler manifold of dimension $d$ and let $\hat\omega_1$, $\hat\omega_2$ be K\"ahler classes on $X$. Then the  Schur class $s_{(1,1,\ldots,1)}(\hat\omega_1,\hat\omega_2)$ of degree $d-2$ has the Hodge-Riemann property.  
\end{proposition}
\begin{proof}
Set $k=d-2$ and let $\omega$ be any K\"ahler form on $X$. We note that if $\omega_1$ and $\omega_2$ are strictly positive $(1,1)$-forms then 
$(-1)^{d-2}s_{d-2}(\omega_1,\omega_2)=\sum_{j=0}^{d-2}\omega_1^{d-2-j}\omega_2^j$ is also strictly positive. By the above consideration our question reduces itself to the corresponding linear algebraic Question \ref{question:LinearAlgebra}.

 So let $V$ be a complex vector space of dimension $d$ and $U$ be a lattice in $V$ as in the above discussion. 
It is then enough to show that $(-1)^{d-2}s_{d-2}(\omega_1,\omega_2)$ has the Hodge-Riemann property for all strictly positive $(1,1)$-forms $\omega_1, \omega_2\in V^{1,1}=V\otimes \bar V$.  

Using harmonic representatives with respect to the flat metric the above question is equivalent to showing that for any  two K\"ahler classes $\hat\omega_1$, $\hat\omega_2$ on the abelian variety $Y:=\C^d/(\Z^d+i\Z^d)$ the Segre class 
$(-1)^{d-2}s_{d-2}(\hat\omega_1,\hat\omega_2)$ has the Hodge-Riemann property. 
If $\hat\omega_1$, $\hat\omega_2$ have integer coefficients, they are the first Chern classes of two ample line bundles $H_1$ and $H_2$ on $Y$.
We consider their direct sum $E:=H_1\oplus H_2$ and the projective bundle $P:=\P_Y(E)$, with projection 
$\pi:\P_Y(E)\to Y$. 
The Chern class $\xi:=c_1(\cO_{\P(E)}(1))$ of the tautological quotient bundle $\cO_{\P(E)}(1)$ on $\P_Y(E)$ is ample and one has $\pi_*(\xi^{j+1})=(-1)^{j}s_j(E)$, for all $j\in\N$, \cite[Section 3.1]{FultonIT}. Thus the quadratic forms $Q_P$ and $Q_Y$ defined on $H^{1,1}(P, \R)$ and on $H^{1,1}(Y, \R)$ respectively by
$$Q_P(\eta):=\int_P\xi\eta^2, \  Q_Y(\alpha):=\int_Y(-1)^{d-2}s_{d-2}(E)\alpha^2$$
compare using the projection formula giving
$$Q_Y(\alpha)=Q_P(\pi^*\alpha).$$
Noting that $Q_P$ has the Hodge-Riemann property, that $\pi^*$ is injective on $H^{1,1}(Y)$ and that $Q_P(\pi^*h)$ is positive for any ample class $h$ on $Y$, we see by the fourth condition of  our Definition-Lemma \ref{deflemma:HRofquadraticform} that $Q_Y$ has the Hodge-Riemann property as well. 
Moreover if  $\hat\omega_1$, $\hat\omega_2$ are integer classes on   $Y$ as above and if $\hat\epsilon$ is any $(1,1)$ class such that  $\hat\omega_1+\hat\epsilon$ and $\hat\omega_2+\hat\epsilon$ lie in the K\"ahler cone of $Y$, the twisted vector bundle $E\langle\hat\epsilon\rangle:=(H_1\oplus H_2)\langle\hat\epsilon\rangle$ is ample on $Y$ and the twisted line bundle $\cO_{\P(E)}(1)\langle\pi^*\hat\epsilon\rangle$ is ample on $P$.  
Thus $(\xi+ \pi^*\hat\epsilon)^{d-1}$ has the 
 Hodge-Riemann property on $P$ and by the same argument as above
 $\pi_*((\xi+ \pi^*\hat\epsilon)^{d-1})$  has the 
 Hodge-Riemann property on $Y$. 
 
 Now a direct computation gives  
 $\pi_*(\xi^{j+1})=(-1)^{j}s_j(E)$ for all $j\in\N$ and 
 $$\pi_*((\xi+ \pi^*\hat\epsilon)^{d-1})=
 \pi_*( \sum_{j=0}^k \binom{k+1}{j+1} \xi^{j+1}( \pi^*\hat\epsilon)^{k-j})= $$
 $$\sum_{j=0}^k \binom{k+1}{j+1} (-1)^{j}s_j(E)\hat\epsilon^{k-j} = 
 (-1)^{d-2}s_{d-2}(E\langle\hat\epsilon\rangle)=
 (-1)^{d-2}s_{d-2}(\hat\omega_1+\hat\epsilon,\hat\omega_2+\hat\epsilon),$$
 cf. \cite[Example 3.1.1]{FultonIT}, hence 
 the class $(-1)^{d-2}s_{d-2}(\hat\omega_1+\hat\epsilon,\hat\omega_2+\hat\epsilon)$ has the Hodge-Riemann property on $Y$. 

Going back to the problem dealing with arbitrary  $(1,1)$-forms $\omega_1, \omega_2\in V^{1,1}$ we remark that by a change of coordinates we may always simultaneously diagonalize $\omega_1$ and $\omega_2$ to obtain 
$\omega_1=i\sum_{j=1}^dd z_j\wedge d \bar z_j$, 
$\omega_2=i\sum_{j=1}^d\lambda_j d z_j\wedge d \bar z_j$ with $\lambda_j>0$. If the coefficients $\lambda_j$ are all rational, we are done. Otherwise let us choose for each $j$ some rational number $\tilde \lambda_j$ close to $\lambda_j$. When   $\lambda_j$ is rational we will take $\tilde \lambda_j$ equal to $\lambda_j$. Put $\tilde \omega_2=i\sum_{j=1}^d\tilde \lambda_j d z_j\wedge d \bar z_j$. 

By what we have  just seen  if $\epsilon$ is any real $(1,1)$-form such that $\omega_1+\epsilon$ and $\tilde \omega_2+\epsilon$ are strictly positive, the form 
$(-1)^{d-2}s_{d-2}(\omega_1+\epsilon,\tilde \omega_2+\epsilon)$ has the Hodge-Riemann property. We set
$\epsilon_j:=\frac{\lambda_j-\tilde \lambda_j}{\lambda_j-1}$ if $\lambda_j\neq1$, and $\epsilon_j:=0$ otherwise. Clearly $\epsilon_j$ tends to zero when $\tilde \lambda_j$  tends to $\lambda_j$.
Moreover $(\tilde \lambda_j-\epsilon_j)(1-\epsilon_j)^{-1}=\lambda_j$ for all $j$.
Consider now the $(1,1)$-form  $\epsilon:= -\sum_{j=1}^d\epsilon_j i d z_j\wedge d \bar z_j$. 
Next we check that we may act on the pair $(\omega_1+\epsilon,\tilde \omega_2+\epsilon)$ again by coordinate change in order to bring it to the form $(\omega_1, \omega_2)$ when written with respect to the new coordinates. This will end the proof of the Proposition. 
If $M(\omega)$ is the hermitian matrix of the coefficients of a real $(1,1)$-form $\omega$, a coordinate change on $\omega$ will transform $M(\omega)$ into 
$\bar P^tM(\omega)P$ where $P$ is the base change matrix. We reach our desired coordinate change by taking $P$ to be the diagonal matrix with  diagonal entries $(1-\epsilon_j)^{-\frac{1}{2}}$ for $j\in\{1,...,d\}$. 
\end{proof}

As above this yields the following linear algebra consequence:

\begin{corollary}\label{cor:Segre}
Let $\omega_1,\omega_2$ be K\"ahler forms on a $d$-dimensional complex vector space $V$.   Then
$$\omega_1^{d-2} + \omega_1^{d-3}\wedge \omega_2 + \cdots + \omega_2^{d-2} \in V^{d-2,d-2}$$
has the Hodge-Riemann property.
\end{corollary}

Finally we observe that an easy consequence of Proposition \ref{prop:Segre} is the following injectivity statement which was first noticed in \cite[Proposition 1.1]{FuXiao} (and in \cite[Proposition 6.5]{GrebToma} in the projective case).
\begin{corollary}\label{cor:injectivity}
Let $K(X)\subset H^{1,1}(X)$ be the K\"ahler cone of a compact K\"ahler manifold $X$ of dimension $d$. Then the map $K(X)\to  H^{d-1,d-1}(X)$, $\hat\omega\mapsto\hat\omega^{d-1}$, is injective. 
\end{corollary}
\begin{proof}
The statement follows directly from the fact that $s_{(1,1,...,1)}(\hat\omega_1,\hat\omega_2)$ has the Lefschetz property when $\hat\omega_1,\hat\omega_2\in K(X)$, noting that
$$\hat\omega_1^{d-1}-\hat\omega_2^{d-1}=
(\hat\omega_1-\hat\omega_2)(\sum_{j=0}^{d-2}\hat\omega_1^{d-2-j}\hat\omega_2^j)=(\hat\omega_1-\hat\omega_2)s_{(1,1,...,1)}(\hat\omega_1,\hat\omega_2).$$
\end{proof}

 \section{Questions and extensions}\label{sec:questions}
 
\subsection{The Hodge-Riemann property for other degrees } 
We have focused purely on the case $|\lambda| = \dim X-2$.    Example \ref{ex: higher codimension} shows that for higher degrees the natural generalization of the Hodge-Riemann property as defined in  \cite{Dinh2}, \cite{Dinh} does not hold for Schur classes of ample vector bundles in general. Nevertheless the following question is natural:
 
\begin{question}
What can be said about the intersection form 
$$\mathcal Q (\alpha,\alpha') = \int_X \alpha {s_{\lambda}(E)} \alpha'  \text{ for } \alpha,\alpha' \in H^{j,j}(X)$$
where $E$ is an ample bundle and $|\lambda| = \dim X-2j$ with $j\ge 2$? 
\end{question}

\begin{example}\label{ex: higher codimension}
Let 
$ X = \mathbb P^2\times \mathbb P^2\times \mathbb P^2$. For $i=0,1,2$ let $\pi_i$ be the projection to the $i$-th factor and let $x_i$ denote the  hyperplane class on each factor.    By the Kunneth formula $H^{2,2}(X)$ has 
$ \{x_0^2, x_1^2, x_2^2, x_1x_2, x_0x_2,  x_0x_1 \}$ as basis.
Now let
$$ E = \pi_0^* \mathcal O_{\mathbb P^2}(1)  \oplus  \pi_1^* \mathcal O_{\mathbb P^2}(1)  \oplus  \pi_2^* \mathcal O_{\mathbb P^2}(1).$$
Then $E$ is nef but not ample and
$ c_2(E) =x_1x_2 +  x_2x_0+ x_0x_1, \ c_1(E) = x_0 + x_1 + x_2.$
For  $t\ge 0$ consider the $\mathbb R$-twisted vector bundle
$$E_t:= E\langle t \det(E)\rangle$$
which is ample for $t>0$. Consider further the intersection forms 
$$\mathcal Q_t(\alpha,\alpha'): = \int_X \alpha c_2(E_t) \alpha ' \text{ for } \alpha,\alpha'\in H^{2,2}(X),$$
$$\mathcal R(\alpha,\alpha'): = \int_X \alpha c_2(E) \alpha ' \text{ for } \alpha,\alpha'\in H^{2,2}(X),$$
$$\mathcal S(\alpha,\alpha'): = \int_X \alpha c_1(E)^2 \alpha ' \text{ for } \alpha,\alpha'\in H^{2,2}(X).$$
Since 
$$c_2(E_t) = c_2(E) + (2t + 3t^2) c_1(E)^2$$
we get 
$$\mathcal Q_t = \mathcal R +  (2t + 3t^2) \mathcal S.$$
One checks by direct calculation that the determinant of the associated matrices with respect to the given basis is negative for $\mathcal R $ and positive for $\mathcal S$.  Hence there is some $t>0$ for which $\mathcal Q_t$ is singular, and thus the Hard Lefschetz property and consequently also the Hodge-Riemann property fail for $c_2(E_t)$.

Note this does not contradict  the Bloch-Gieseker Theorem \ref{thm:blochgiesekerI} since $E_t$ has rank 3.
\end{example}
 
\subsection{Combinations of Schur Classes} \label{sec:convexcombinations} Using the material in \cite[3c]{FultonLazarsfeld} one can extend our main result easily to monomials of Schur classes of possibly different ample bundles.  

To see this, let $E_1,\ldots, E_r$ be ample bundles on a projective manifold $X$ and $\lambda_1,\ldots,\lambda_r$ be partitions with $\sum_{j=1}^r |\lambda_i| = d-2$.  Suppose $\rank(E_j)\ge |\lambda_j|$ for all $j$.  Then   for each $j=1,\ldots,r$ we can construct just as in section \ref{sec:Schur} a cone $C_j\subset \Hom(V_j,E_j)=:F_j$  where $V_j$ is a fixed vector space.  Since each $C_j$ is flat over $X$ there is a product cone 
$$ C: = \Pi_{j=1}^r C_j\subset \oplus_{j=1}^r F_j$$
with the property that
$$\int_{X} \alpha \left( \Pi_{j=1}^r s_{\lambda_j}(E_j)\right) \alpha' = \int_{[C]} (\pi^* \alpha)c_{N-2}(U) (\pi^* \alpha')$$
where $U$ is the tautological bundle on $\mathbb P(F) = \mathbb P(\oplus_{j=1}^r F_j)$ and $N: = \dim C$.  Thus Theorem \ref{thm:maincone} implies that the class
 $$\Pi_{j=1}^r s_{\lambda_j}(E_j) \in H^{d-2,d-2}(X)$$
 has the Hodge-Riemann property.  Observe in particular that if each $E_j$ has rank $1$ we get that
 $$c_1(E_1)\wedge \cdots \wedge c_{1}(E_{d-2})$$
 has the Hodge-Riemann property, as proved by Gromov in the K\"ahler case (see Remark \eqref{rmk:HRandLgeneralremarks}(4)).

 \begin{remark}\label{rem:convexcombinations}
Note that  arbitrary convex combinations of monomials of Schur classes of several ample vector bundles bundles need not have the Hodge-Riemann property. Indeed this can already be seen for a combination of the type $c_1^2(L_1)+ac_1^2(L_2)$, where $L_1$, $L_2$ are ample line bundles on a $4$-dimensional abelian variety. An example is obtained by taking $d=4$ and $\omega_1$, $\omega_2$ as in the proof of Proposition \ref{prop:Segre} with $\lambda_1=\lambda_2=\frac{1}{7}$, $\lambda_3=\lambda_4=2$ and by considering $\Omega_a:=\omega_1^2 +a\omega_2^2$. Then the bilinear form $(\alpha,\alpha')\mapsto \int\alpha \wedge \Omega_a\wedge  \alpha'$ on $T^{1,1}_\R$ has signature $(1,15)$ for $a\in[0,3[\cup]\frac{49}{12},\infty]$, is degenerate for  $a\in\{3, \frac{49}{12}\}$ and has signature $(2,14)$ for $a\in]3, \frac{49}{12}[$. 
 \end{remark}

\begin{question}\label{question:convex}
Is it possible to describe the collection of tuples $\{a_{\lambda}\}$ of non-negative numbers such that 
$$\sum_{|\lambda|=d-2} a_{\lambda} s_{\lambda}(E)$$
 has the Hodge-Riemann property for all ample vector bundles $E$ of rank at least $d-2$?
  \end{question}

The only case we can answer this completely is when $d=4$.   For then there are two Schur classes, $c_2$ and $c_1^2-c_2$, and we know that $c_2(E\langle tc_1(E)\rangle$ and $(c_1^2-c_2)(E\langle tc_1(E))$ have the Hodge-Riemann property for all $t\ge 0$.  Together these imply that any convex combination of $c_2(E)$ and $(c_1^2-c_2)(E)$ has the Hodge-Riemann property.

The following example shows that  in higher dimension there can be some constraint on the $a_{\lambda}$ (beyond requiring them to be all non-negative).   Let $X=\mathbb P^2\times \mathbb P^3$  Then $N^1(X)$ is two-dimensional, with generators $a,b$ that satisfy $a^3=0$, $a^2b^3=1$.
Set
$\mathcal O_X(a,b) = \mathcal O_{\mathbb P_2}(a) \boxtimes \mathcal O_{\mathbb P^3}(b)$ and consider the nef vector bundle 
$$ E = \mathcal O(1,0) \oplus \mathcal O(1,0) \oplus \mathcal O(0,1).$$
Then an elementary computation, left to the reader, shows that the class
$$ (1-t) c_3(E) + t s_{(1,1,1)}(E)$$
gives an intersection form on $N^1(X)$ with matrix 
$$Q_t:= \left(\begin{array}{cc} t&2t \\ 2t&1+2t \end{array}\right).$$
One observes that for $t\in (0,1/2)$ the matrix $Q_t$ has two strictly positive eigenvalues.   Thus fixing $t\in (0,1/2)$, any small pertubation of $E$ by an ample class gives an ample $\mathbb R$-twisted bundle $E'$ so that $(1-t)c_3(E') + t s_{(1,1,1)}(E')$ does not have the Hodge-Riemann property.

\subsection{The non-projective case}   Assume $X$ is a K\"ahler manifold of dimension $d$.  Then Demailly-Peternel-Schneider \cite[Proposition 2.3]{DemaillyPeternellSchneider} has shown that for any nef vector bundle $E$ on $X$ the non-strict inequality
 $$\int_X s_{\lambda}(E) \ge 0$$
holds for any partition with $|\lambda|=d$. 

 \begin{question}
What can be said for $s_{\lambda}(E)$ when $|\lambda|= d-2$ and $X$ is K\"ahler of dimension $d$ but non-projective.     For instance is there a version of the Hodge-Index inequality \eqref{eq:hodgeindexcn-2definitive}, or the related inequalities \eqref{eq:hodgeindexiii}, \eqref{eq:logconcaveI} for $E$ nef in the K\"ahler setting?
 \end{question}

\subsection{Borderline case for the higher-rank Khovanskii-Teiss\-ier} 
An easy consequence of the Hodge-Index Theorem \cite[Theorem 6.2]{Voisin} is that if $\alpha,\beta\in H^{1,1}(X,\mathbb R)$ are ample, and are on the borderline of the Khovanskii-Teissier inequality (by which we mean the function $i\mapsto \log\int_X \alpha^i \beta^{d-i}$ is affine) then $\alpha$ and $\beta$ are proportional.  Teissier asks \cite[p96]{TeissierBonnesen} if this remains true when $\alpha,\beta$ are merely nef and big, which has been answered positively by Boucksom-Favre-Jonsson \cite{BoucksomFavreJonsson},  Cutkosky \cite{Cutkosky} and Fu-Xiao \cite{FuXiaoProportionality}.

\begin{question}
Can one characterize those  nef vector  bundles $E$ such that the map $i\mapsto\log\int_X c_i(E) c_1(E)^{d-i}$ is affine?
\end{question}


\providecommand{\bysame}{\leavevmode\hbox to3em{\hrulefill}\thinspace}
\providecommand{\MR}{\relax\ifhmode\unskip\space\fi MR }
\providecommand{\MRhref}[2]{%
  \href{http://www.ams.org/mathscinet-getitem?mr=#1}{#2}
}
\providecommand{\href}[2]{#2}

\end{document}